\documentclass[a4paper]{article}

\usepackage[T1]{fontenc}
\usepackage[utf8]{inputenc}
\usepackage[pass]{geometry}
\usepackage{tikz}
	\usetikzlibrary{positioning}
\usepackage{amsmath}
\usepackage{amsfonts}
\usepackage{amssymb}
\usepackage{mathtools}
\usepackage{amsthm}
\usepackage{etoolbox}
\usepackage{caption}
\RequirePackage{hyperref}
	\hypersetup{linktoc=all,hidelinks}
\usepackage[capitalize,nameinlink]{cleveref}
\usepackage[noadjust]{cite}

\theoremstyle{definition}
\newtheorem{theorem}{Theorem}[section]
\newtheorem{lemma}[theorem]{Lemma}
\newtheorem{proposition}[theorem]{Proposition}
\newtheorem{conjecture}{Conjecture}
\newtheorem{claim}{Claim}
\newtheorem{case}{Case}[theorem]
\newtheorem{subcase}{Case}[case]
\newtheorem{step}{Step}[section]
\newtheorem{example}{Example}[section]
\newtheorem*{definition}{Definition}
\theoremstyle{remark}
\newtheorem{remark}{Remark}[section]
\crefname{claim}{Claim}{Claims}
\crefname{case}{Case}{Cases}
\crefname{subcase}{Case}{Cases}
\crefname{conjecture}{Conjecture}{Conjectures}
\AfterEndEnvironment{case}{\noindent\ignorespaces}
\AfterEndEnvironment{subcase}{\noindent\ignorespaces}
\newcounter{parenum}[theorem]
\newcommand{\parenum}{\par\noindent\stepcounter{parenum}\textbf{(\roman{parenum}) }}

\let\eqref\labelcref
\crefname{equation}{}{}

\crefname{enumi}{}{}

\makeatletter
\newenvironment{boldproof}[1][\proofname]{\par
  \pushQED{\qed}%
  \normalfont \topsep6\p@\@plus6\p@\relax
  \trivlist
  \item[\hskip\labelsep
        \bfseries
    #1\@addpunct{.}]\ignorespaces
}{%
  \popQED\endtrivlist\@endpefalse
}
\makeatother

\newcommand{\pathto}[2]{\ensuremath{(#1,#2)}}
\newcommand{\dotsp}{\cdots}
\newcommand{\dir}[1]{\smash{\vec{#1}}}
\newlength{\revdirraise}
\newlength{\revdirextraraise}
\newcommand{\revdir}[1]{\smash{%
	\settoheight{\revdirraise}{$#1$}%
	\settodepth{\revdirextraraise}{$#1$}%
	\ooalign{$#1$\cr{\raisebox{\revdirraise+\revdirextraraise+0.85pt}%
		{\rlap{$\mkern4.1mu$\rotatebox[origin=c]{180}%
			{$\vec{\phantom #1}$}}}}}}}
\let\overrightarrow\dir
\let\overleftarrow\revdir
\newcommand{\open}{^\circ}

\DeclarePairedDelimiter{\card}{\lvert}{\rvert}
\DeclarePairedDelimiter{\set}{\lbrace}{\rbrace}
\DeclarePairedDelimiterX{\defset}[2]{\lbrace}{\rbrace}
	{\,#1:#2\,}

\AtBeginDocument{\frenchspacing}

\newcommand{\internalvertex}[2]{node [circle, inner sep=0pt, outer sep=0pt, minimum size=4pt, fill=#1, draw=#2] {}}
\newcommand{\vertex}{\internalvertex{black}{black}}

\newcommand{\point}{\node [inner sep=0pt,outer sep=0pt,minimum size=0pt]}

\newcommand{\ellipsisdot}{node [circle, inner sep=0pt, outer sep=0pt, minimum size=2.5pt, fill=black, draw=black] {}}

\def\upstrut(#1,#2){\node at (0,#2) [anchor=south] {\vphantom{#1}}}
\def\downstrut(#1,#2){\node at (0,#2) [anchor=north] {\vphantom{#1}}}

\newcommand{\figtoolocaldirac}{%
\begin{tikzpicture}
\foreach \x in {-2,0,2}{
	\draw
		(\x,1) -- (\x,0) -- (\x+1,1)
			-- (\x+1,0) -- (\x,1) -- (\x+1,1)
		(\x,0) -- (0,-1.5) -- (\x+1,0)
			-- (1,-1.5) -- (\x,0) -- (\x+1,0)
		(\x,0) \vertex
		(\x,1) \vertex
		(\x+1,0) \vertex
		(\x+1,1) \vertex;}
\draw
	(0,-1.5) -- (1,-1.5)
	(0,-1.5) \vertex
	(1,-1.5) \vertex;
\end{tikzpicture}}

\newcommand{\figthreeballnottwoconnected}{%
\begin{tikzpicture}[scale=1.4]
\point (a) at (0.1,0) {};
\draw [overlay] (-.1,-2.5)
	\foreach \x/\y in {1/55,2/75,3/95,4/115}{
		+(\y:2) node (b\x) [inner sep=0pt,outer sep=0pt] {}};
\draw [overlay] (0.1,0.3)
	\foreach \x/\y in {1/-65,2/-85,3/-105,4/-125}{
		+(\y:2) node (c\x) [inner sep=0pt,outer sep=0pt] {}};
\draw [overlay] (-.2,-4.3)
		+(120:2) node (d) [inner sep=0pt,outer sep=0pt] {};
\draw [overlay] (.55,-4.5)
	\foreach \x/\y in {1/50,2/70,3/90,4/110}{
		+(\y:2) node (d\x) [inner sep=0pt,outer sep=0pt] {}};
\draw [overlay] (0.3,-1.8)
	\foreach \x/\y in {1/-60,2/-80,3/-100,4/-120}{
		+(\y:2) node (e\x) [inner sep=0pt,outer sep=0pt] {}};
\foreach \x in {1,2,3,4}{
	\draw (a) -- (b\x);
	\draw (c1) -- (d\x);
	\draw (c\x) -- (d);
	\draw (d) -- (e\x);
	\foreach \y in {1,...,4}{
		\draw (b\x) -- (c\y);
		\draw (d\x) -- (e\y);}
	\foreach \y in {\x,...,4}{
		\draw (b\x) -- (b\y);
		\draw (c\x) -- (c\y);
		\draw (d\x) -- (d\y);
		\draw (e\x) -- (e\y);}}
\draw (a) \vertex node [above right] {$v$};
\draw (d) \vertex;
\foreach \x in {1,2,3,4}{
	\draw (b\x) \vertex;
	\draw (c\x) \vertex;
	\draw (d\x) \vertex;
	\draw (e\x) \vertex;}
\end{tikzpicture}}

\newcommand{\figlocaldirac}{%
\begin{tikzpicture}[yscale=0.5]
\foreach \x in {-0.8,0.8}{
	\foreach \y in {-1.5,1.5}{
		\draw (0,\x) -- (1,\y);
		\draw (4,\y) -- (5,\x);}}
\foreach \x in {-1.5,1.5}{
	\foreach \y in {-1,0,1}{
		\draw (1,\x) -- (2,\y);
		\draw (3,\y) -- (4,\x);}}
\foreach \x in {-1,0,1}{
	\foreach \y in {-1,0,1}{
		\draw (2,\y) -- (3,\x);}}
\draw (0,-0.8) -- (0,0.8);
\draw (5,-0.8) -- (5,0.8);
\draw (2,-1) -- (2,1) to [out=255,in=105] (2,-1); 
\draw (3,-1) -- (3,1) to [out=285,in=75] (3,-1); 
\foreach \x in {-0.8,0.8}{
	\draw (0,\x) \vertex;
	\draw (5,\x) \vertex;}
\foreach \x in {-1.5,1.5}{
	\draw (1,\x) \vertex;
	\draw (4,\x) \vertex;}
\foreach \x in {-1,0,1}{
	\draw (2,\x) \vertex;
	\draw (3,\x) \vertex;}
\foreach \x in {-0.7,0,0.7}{
	\draw (1,\x) \ellipsisdot;
	\draw (4,\x) \ellipsisdot;}
\draw [rounded corners=0.25cm] (0.75,-2) rectangle (1.25,2);
\draw [rounded corners=0.25cm] (3.75,-2) rectangle (4.25,2);
\draw (1,-2) node [anchor=north] {$K_6$};
\draw (4,-2) node [anchor=north] {$K_6$};
\end{tikzpicture}}

\hypersetup{%
	pdftitle={A localization method in Hamiltonian graph theory},
	pdfauthor={Armen S. Asratian, Jonas B. Granholm, Nikolay K. Khachatryan},
	pdfkeywords={Hamilton cycle, local conditions, infinite graphs, Hamilton curve}}

\begin{document}

\title{A localization method in Hamiltonian graph theory}
\author{Armen S. Asratian%
	\footnote{Department of Mathematics, Linköping University, email: armen.asratian@liu.se},
	Jonas B. Granholm%
	\footnote{Department of Mathematics, Linköping University, email: jonas.granholm@liu.se},
	Nikolay K. Khachatryan%
	\footnote{Synopsys Armenia CJSC, email: nikolay@synopsys.com}}
\date{}
\maketitle


\begin{abstract}
\noindent
The classical global criteria for the existence
of Hamilton cycles only apply to
graphs with large edge density and small diameter.
In a series of papers Asratian and Khachatryan
developed local criteria for the existence
of Hamilton cycles in finite connected graphs,
which are analogues of the classical global criteria due to
Dirac~(1952)\nocite{dirac52}, Ore~(1960)\nocite{ore60},
Jung~(1978)\nocite{jung78}, and Nash-Williams~(1971)\nocite{nash-williams71}.
The idea was to show that the global concept of Hamiltonicity can,
under rather general conditions,
be captured by local phenomena, using the structure of balls of small radii.
(The ball of radius~$r$ centered at a vertex~$u$ is a subgraph of~$G$ induced by
the set of vertices whose distances from~$u$ do not exceed~$r$.)
Such results are called \emph{localization theorems}
and present a possibility to extend known classes
of finite Hamiltonian graphs.

In this paper we formulate a general approach for finding localization theorems
and use this approach to formulate local analogues of well-known results of
Bauer et~al.~(1989)\nocite{bauer89}, Bondy~(1980)\nocite{bondy80},
H\"aggkvist and Nicoghossian (1981)\nocite{haggkvist81},
and Moon and Moser~(1963)\nocite{moon63}.
Finally we extend two of our results to infinite locally finite graphs
and show that they guarantee the existence of Hamiltonian curves,
introduced by K\"undgen, Li and Thomassen~(2017)\nocite{kundgen17}.
\end{abstract}

\bgroup
\noindent
\setlength{\parfillskip}{0pt}%
\par
\egroup


\section{Introduction}

We use~\cite{diestel} for terminology and notation not defined here,
and consider graphs without loops and multiple edges only.
A graph~$G$ is called \emph{locally finite} if every vertex of~$G$ has finite degree.
A graph~$G$ is finite or infinite according to the number of vertices in~$G$.

Let $V(G)$ and $E(G)$ denote, respectively, the vertex set and edge set of a graph~$G$,
and let $d_G(u)$ or simply $d(u)$ denote the degree of a vertex~$u$.
The distance between vertices $u$ and $v$ in~$G$
is denoted by $d_G(u,v)$ or simply $d(u,v)$.
The greatest distance between any two vertices in~$G$ is the \emph{diameter} of~$G$.
For each vertex $u\in V(G)$ we denote by $N_r(u)$ and $M_r(u)$
the set of all vertices $v\in V(G)$ with $d(u,v)=r$ and $d(u,v)\leq r$, respectively.
The set $N_1(u)$ is usually denoted by $N(u)$.
The \emph{ball of radius~$r$ centered at~$u$}, denoted by $G_r(u)$, is the subgraph of~$G$ induced by the set $M_r(u)$.

A \emph{Hamilton cycle} of a finite graph~$G$ is a cycle containing every vertex of~$G$.
A graph that has a Hamilton cycle is called \emph{Hamiltonian}.
There is a vast literature in graph theory devoted to
obtaining sufficient conditions for Hamiltonicity
(see, for example, the surveys~\cite{gould03,gould14}).
Almost all of the existing sufficient conditions for a finite graph~$G$ to
be Hamiltonian contain some global parameters of~$G$ (such as the number of vertices)
and only apply to graphs with large edge density (\,$|E(G)|\geq \text{constant}\cdot|V(G)|^2$\,)
and/or small diameter (\,$o(|V(G)|)$\,).

The following two classical theorems are examples of such results:
\begin{theorem}[Dirac~\cite{dirac52}]
\label{oldthm:dirac}
A finite graph~$G$ on at least three vertices is Hamiltonian if
$d(u)\geq |V (G)|/ 2$ for every vertex~$u\in V(G)$.
\end{theorem}
\begin{theorem}[Ore~\cite{ore60}]
\label{oldthm:ore}
A finite graph~$G$ on at least three vertices is Hamiltonian if
$d(u) + d(v)\geq |V(G)|$ for every pair of nonadjacent vertices $u,v\in V(G)$.
\end{theorem}
Graphs satisfying these conditions are called \emph{Dirac graphs} and \emph{Ore graphs}, respectively.

Another type of sufficient conditions for Hamiltonicity of a finite graph~$G$,
which contains no global parameter of~$G$,
was found by Oberly and Sumner~\cite{oberly79}.
A graph is called \emph{claw-free} if it has no induced subgraph isomorphic to $K_{1,3}$,
and \emph{locally connected} if
the subgraph induced by the set $N(u)$ of neighbors of $u$ is connected
for each vertex $u$.

\begin{theorem}[Oberly and Sumner~\cite{oberly79}]
\label{oldthm:oberly}
Every finite, connected, locally connected, claw-free graph on at least three vertices is Hamiltonian.
\end{theorem}

Note that this result can be formulated in terms of balls as follows:
A finite connected graph~$G$ on at least three vertices is Hamiltonian if
for every vertex $u\in V(G)$ the ball
$G_1(u)$ satisfies the condition $\kappa\bigl(G_1(u)\bigr)\geq 2\geq \alpha\bigl(G_1(u)\bigr)$,
where $\kappa\bigl(G_1(u)\bigr)$ and $\alpha\bigl(G_1(u)\bigr)$ denotes
the vertex connectivity and independence number of the ball~$G_1(u)$, respectively.

Asratian and Khachatryan showed in
\cite{asratyan85,hasratian90,asratian06,asratian07}
that also some classical sufficient conditions for Hamiltonicity of graphs
that contain global parameters (e.g. \cref{oldthm:dirac,oldthm:ore}), can be reformulated
in such a way that
every global parameter in those
conditions
is replaced by a parameter of a ball with small radius.
Such results are called \emph{localization theorems} and present a possibility
to extend known classes of Hamiltonian graphs.
For example, the following generalization of Dirac's theorem and Ore's theorem was obtained in \cite{hasratian90}
(see also \cite[Thm.~10.1.3]{diestel}):

\begin{theorem}[Asratian and Khachatryan~\cite{hasratian90}]
\label{oldthm:L0}
\label{prop:localoreL0}
Let~$G$ be a connected finite graph on
at least three vertices such that
\[d(u)+d(v)\geq |N(u)\cup N(v)\cup N(x)|\]
for every path $uxv$ with $uv\notin E(G)$.
Then $G$ is Hamiltonian.
\end{theorem}

Note that in contrast with Dirac's and Ore's theorems,
\cref{oldthm:L0} applies to infinite classes of graphs~$G$ with large diameter
(\,$\geq \text{constant}\cdot|V(G)|$\,)
and small edge density
(\,$|E(G)|\leq \text{constant}\cdot|V(G)|$\,).

In 2004, Diestel and K\"uhn~\cite{diestel04a} suggested a new concept
for infinite locally finite graphs called \emph{Hamilton circles},
which are analogues of Hamilton cycles in finite graphs.
Diestel~\cite{diestel10} launched the ambitious project of
extending results on finite Hamilton cycles to Hamiltonian circles.
Georgakopoulos~\cite{georgakopoulos09} showed that
the square of a 2-connected, infinite, locally finite graph $G$ has a Hamilton circle,
extending Fleishner's theorem~\cite{fleischner74} for finite graphs.
Heuer \cite{heuer15} and Hamann et~al.~\cite{hamann16} showed that
every connected, locally connected, infinite, locally finite, claw-free graph has a Hamilton circle,
extending \cref{oldthm:oberly}.
Bruhn and Yu~\cite{bruhn08} obtained some other results concerning
an extension of Tutte's theorem on 4-connected planar graphs.

In each of these results the condition for an infinite graph is
the same as the condition in the corresponding result for finite graphs.
The situation is different for classical theorems on Hamilton cycles with conditions involving
the number of vertices of a graph.
In order to extend those theorems to infinite graphs we need some local analogues.
For example,
Diestel~\cite{diestel10} conjectured that every infinite locally finite graph
satisfying the condition of \cref{prop:localoreL0} (a local analogue of Ore's theorem)
has a Hamilton circle,
and Heuer \cite{heuer16} proved this conjecture for a class of graphs.

In the present paper we formulate a general method for finding
local analogues of classical theorems with global parameters,
and apply this method to formulate local analogues of well-known results of
Bauer et~al.~\cite{bauer89}, Bondy~\cite{bondy80}, Moon and Moser~\cite{moon63},
and H\"aggkvist and Nicoghossian~\cite{haggkvist81}.
Furthermore we extend two of our results to infinite locally finite graphs
and show that they guarantee the existence of Hamiltonian curves,
introduced by K\"undgen, Li and Thomassen~\cite{kundgen17}.
We believe that our localization method can be useful for
extending many other criteria for Hamiltonicity of finite graphs to infinite graphs.


\section{ Definitions and preliminary results}

Let~$G$ be a connected graph and $v$ a vertex in a ball
$G_r(u), r\ge1$. We call $v$ an \emph{interior vertex} of $G_r(u)$ if the ball $G_1(v)$ is a
subgraph of $G_r(u)$. Clearly, every vertex in $G_{r-1}(u)$ is interior for $G_r(u)$ and if
$G_r(u)=G$ then all vertices in~$G$ are interior vertices.

The \emph{connectivity} $\kappa(G)$ of a graph~$G$
is the smallest number of vertices whose removal
turns $G$ into a disconnected or trivial graph.

Let~$C$ be a cycle of a graph~$G$.
We denote by $G-C$ the graph induced by the set $V(G)\setminus V(C)$.
Furthermore, we denote by $\overrightarrow C$ the cycle~$C$ with a given orientation,
and by $\overleftarrow C$ the cycle~$C$ with the reverse orientation.
If $u,v\in V(C)$ then $u\overrightarrow Cv$ denote the consecutive vertices of~$C$
from $u$ to $v$ in the direction specified by $\overrightarrow C$.
The same vertices in reverse order are given by $v\overleftarrow Cu$.
We use $u^+$ to denote the successor of~$u$ on $\overrightarrow C$
and $u^-$ to denote its predecessor.
This notation is extended to sets:
if $S\subseteq V(C)$ then $S^+=\defset{x^+}{x\in S}$.
Analogous notation is used with respect to paths instead of cycles.

Let $G$ be an infinite locally finite graph.
A one-way infinite path in~$G$ is called a \emph{ray}.
We define an equivalence relation on the set of rays in $G$ by saying that
two rays are equivalent if no finite set of vertices separate them%
\footnote{This means that for every finite vertex set $S\subset V(G)$ both rays
have a tail (subray) in the same component of $G-S$.}
in~$G$.
The equivalence classes of this relation are called \emph{ends} of~$G$.
Every end can be viewed as a particular ``point of infinity''.
The \emph{Freudenthal compactification}~$|G|$ of~$G$ is a topological space obtained by taking~$G$,
seen as a 1-complex, and adding the ends of~$G$ as additional points.
For the precise definition of~$|G|$ see \cite{diestel}.
It should be pointed out that, inside~$|G|$,
every ray of~$G$ converges to the end of~$G$ it is contained in.
A \emph{closed curve} in~$|G|$ is the image of an continuous map
from the unit circle~$S^1\subset\mathbb{R}^2$ to~$|G|$.
\medskip

Let $\mathcal{P}$ be some property of a finite graph
(for example, being Hamiltonian or having a dominating cycle).
The following result gives a trivial way to reformulate sufficient conditions for
$\mathcal{P}$
in terms of balls.

\begin{proposition}
\label{thm:triviallocalization}
Let $K$ be a sufficient condition for finite graphs to have the property~$\mathcal{P}$,
and let $A(K,p)$ denote the set of all graphs with at most $p$~vertices
that satisfy the condition~$K$.
If $f(K,p)$ is a function such that
the diameter of any graph $G\in A(K,p)$ does not exceed $f(K,p)$,
then the following two statements are equivalent:
\begin{enumerate}
\item A graph $G$ on $p\geq 3$ vertices has property~$K$,
\label{thm:triviallocalization:(i)}
\item In a connected graph~$G$ on $p\geq 3$ vertices the ball $G_f(u)$ of radius $f=f(K,p)$
satisfies the condition~$K$ for each vertex~$u$.
\label{thm:triviallocalization:(ii)}
\end{enumerate}
\end{proposition}

\begin{proof}
Clearly, \cref{thm:triviallocalization:(i)} implies \cref{thm:triviallocalization:(ii)}.
Suppose now that \cref{thm:triviallocalization:(ii)} holds for~$G$.
We will show that $G_f(u)=G$ for some vertex~$u$,
which implies that \cref{thm:triviallocalization:(i)} holds.
Suppose to the contrary that $G_f(u) \not=G$ for every $u\in V(G)$.
Then there are vertices at distance $f+1$ in $G$.
Let $P=v_0v_1\dotsp v_{f+1}$ be a path in $G$ with $d_G(v_0,v_{f+1})=f+1$,
and consider the ball $G_f(v_1)$.
Then $d_{G_f(v_1)}(v_0,v_{f+1})=f+1$,
since $P$ is a shortest path between $v_0$ and $v_{f+1}$ in $G$ and $P$ is in $G_f(v_1)$.
On the other hand,
$G_f(v_1)$ satisfies~$K$ and the number of vertices of $G_f(v_1)$ is less than~$p$.
Therefore $G_f(v_1)\in A(K,p)$ and thus $d_{G_f(v_1)}(v_0,v_{f+1})\leq f$; a contradiction.
Thus $G_f(u)=G$ for some~$u$, which implies that \cref{thm:triviallocalization:(i)} holds.
\end{proof}

Consider three examples using \cref{thm:triviallocalization}:

\begin{example}
\label{ex:triviallocaldirac}
The diameter of a Dirac graph can not exceed~$2$. Therefore Dirac's theorem has the following equivalent formulation:
\begin{quote}
A finite connected graph~$G$ on at least three vertices is Hamiltonian
if every ball of radius~$2$ in~$G$ is a Dirac graph,
that is, for every vertex $u\in V(G)$
the condition $d_{G_2(u)}(v)\geq {1\over 2}|M_2(u)|$ holds
for each vertex $v\in M_2(u)$.
\end{quote}
\end{example}

\begin{example}
\label{ex:triviallocalore}
The diameter of an Ore graph can not exceed~$2$. Therefore Ore's theorem has the following equivalent formulation:
\begin{quote}
A finite connected graph~$G$ on at least three vertices is Hamiltonian if every ball of radius~$2$ in~$G$ is an Ore graph.
\end{quote}
\end{example}

\begin{example}
The following result is well-known:
\end{example}

\begin{theorem}[Chvátal and Erdős~\cite{chvatal72}]
\label{oldthm:chvatal}
A finite graph~$G$ on at least three vertices is Hamiltonian if $\kappa (G)\geq \alpha (G)$.
\end{theorem}

We will show that the diameter of a graph~$G$ on $p\geq 3$ vertices
satisfying the condition $\kappa(G)\geq \alpha(G)$ does not exceed $\lfloor\sqrt {2p-3}\rfloor$.
Let~$d$ be the diameter of~$G$ and let $x$ and $y$ be two vertices in~$G$ with $d(x,y)=d$.
We have $V(G)=\bigcup_{i=1}^dN_i(x)\cup \{x\}$. Clearly, $|N_i(x)|\geq \kappa(G)$ for $i=1,\dotsc,d-1$.
Therefore
\[|V(G)|-2\geq (d-1)\kappa(G)\geq (d-1)\alpha(G)\geq \tfrac{1}{2}(d-1)(d+1)=\tfrac{1}{2}(d^2-1),\]
because $\kappa(G)\geq \alpha(G)$ and $\alpha(G)\geq {d+1\over 2}$. Thus $p-2\geq {1\over 2}(d^2-1)$
which implies $2p-3\ge d^2$ and $d\leq \lfloor\sqrt {2p-3}\rfloor$.

Thus for any graph~$G$ satisfying \cref{oldthm:chvatal}, the diameter of~$G$ does not exceed $\lfloor\sqrt {2p-3}\rfloor$.%
\footnote{In fact, for every integer $n\geq 3$ there is a graph with $p=2n^2-1$ vertices that satisfies Chvátal--Erdős condition and has diameter $\lfloor\sqrt {2p-3}\rfloor$.
Consider, for example, the graph~$H$ with $V(H)=\bigcup_{i=1}^{2n}V_i$
where $V_1,\dotsc,V_{2n}$ are pairwise disjoint sets, $|V_1|=\dotsb=|V_{2n-1}|=n$,
$|V_{2n}|=n-1$ and two vertices in~$H$ are adjacent if and only if
they both belong to $V_i\cup V_{i+1}$ for some $i\in \{1,2,\dotsc,2n-1\}.$
Clearly, the diameter of~$H$ is $2n-1$ and $ \alpha(H)=\kappa(H)=n$.
Furthermore, $2n>\sqrt {4n^2-5}>2n-1$.
Therefore $\lfloor\sqrt {2p-3}\rfloor=\lfloor\sqrt {4n^2-5}\rfloor=2n-1$.}
Then, by \cref{thm:triviallocalization}, an equivalent formulation of \cref{oldthm:chvatal} in
terms of balls is the following:

\begin{proposition}
A finite connected graph~$G$ on $p\geq 3$ vertices is Hamiltonian if $\kappa(G_r(u))\geq \alpha(G_r(u))$
for every vertex $u\in V(G)$, where $r=\lfloor\sqrt {2p-3}\rfloor$.
\end{proposition}

The following problem arises naturally:
Is a finite connected graph~$G$ on at least three vertices Hamiltonian
if for some integer $t\geq 1$
the condition $\kappa(G_t(u))\geq \alpha(G_t(u))$ holds
for every $u\in V(G)$?
This problem is still open.
Some results on this problem for $t=1$ were obtained in \cite{chen13}.

Some criteria for Hamiltonicity can have different formulations in terms of balls.
Consider, for example, Dirac's theorem again. In \cref{ex:triviallocaldirac} we gave an equivalent
formulation of it in terms of balls. Consider now another formulation.

Let~$G$ be a Dirac graph. Then the diameter of~$G$ does not exceed~$2$ and, therefore,
$M_2(u)=V(G)$ for every vertex $u\in V(G)$.
It seems, therefore, that Dirac's theorem should be equivalent to the following proposition:
A finite connected graph~$G$ on at least three vertices is Hamiltonian
if $d(u)\geq {1\over 2}|M_2(u)|$ for every vertex $u\in V(G)$.
However the graph~$G$ in \cref{fig:toolocaldirac} satisfies this condition but is not Hamiltonian.
\begin{figure}
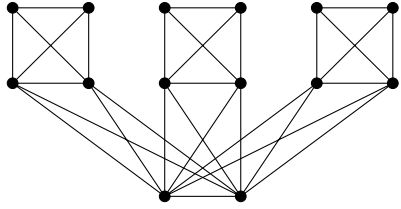

\begin{center}
\figtoolocaldirac
\end{center}
\caption{A non-Hamiltonian graph~$G$ such that $d(u)\ge\tfrac12\card{M_2(u)}$ for every vertex $u\in V(G)$.}
\label{fig:toolocaldirac}
\end{figure}
Surprisingly, Dirac's theorem is equivalent to the following proposition:
\begin{proposition}[\cite{asratian07}]
\label{prop:localdiracM4}
A finite connected graph~$G$ on at least three vertices is Hamiltonian if $d(u)\geq {1\over 2}|M_4(u)|$
for every vertex $u\in V(G)$.
\end{proposition}

In \cref{ex:triviallocalore} we gave a ball formulation of Ore's theorem. Another equivalent formulation of Ore's theorem was given in \cite{asratian07}:

\begin{proposition}[\cite{asratian07}]
\label{prop:localoreM3}
Let~$G$ be a finite connected graph on
at least three vertices such that
$d(u)+d(v)\geq |M_3(x)|$
for every vertex $x\in V(G)$
and for all pairs of non-adjacent interior
vertices $u,v$ of the ball $G_3(x)$.
Then $G$ is Hamiltonian.
\end{proposition}


\section{A general method of localization}

Now we will describe a method that can be used to find local analogues of many criteria for Hamiltonicity
and other properties of finite graphs.

Let $K$ be a sufficient condition for a finite graph to have some property~$\mathcal{P}$
(for example, being Hamiltonian or having a dominating cycle),
that contains a global parameter of the graph
(for example, the number of vertices).

\begin{step}
Find an equivalent formulation (or a variation) $K'$ of the criterion~$K$ and a function $r=r(K',p)$ such that
a connected graph~$G$ on $p\geq 3$ vertices satisfies the criterion~$K'$
if and only if
every ball of radius $r(K',p)$ in $G$ satisfies~$K'$.
\end{step}

\begin{step}
Replace the balls of radius $r(K',p)$ in the condition $K'$ with balls of radius $r(K',p)-1$ or less if this new condition
still guarantees that finite graphs have the property~$\mathcal{P}$.
\end{step}

\begin{step}
Try to relax the condition using the structure of balls.
\end{step}

In other words, we try to replace a global parameter of the graph by parameters of balls of smaller radii.

In \cref{sec:bondy,sec:kappa,sec:MM} we will use this method to find
local analogues of four well-known results in Hamiltonian graph theory.
But first we would like to demonstrate how it works
by applying it to Dirac's and Ore's criteria and rediscovering some old theorems.
Consider, for example, the equivalent formulation of Dirac's theorem
using balls of radius~4,
given in \cref{prop:localdiracM4}.
By decreasing the radius of the balls to 3 we obtain the following result:
\begin{proposition}[\cite{hasratian90}]
\label{prop:localdiracM3}
A finite connected graph~$G$ on at least three vertices is Hamiltonian if $d(u)\geq {1\over 2}|M_3(u)|$
for every vertex $u\in V(G)$.
\end{proposition}

This is a generalization of Dirac's theorem.
There is an infinite class of graphs of diameter~5 that satisfy the
condition of \cref{prop:localdiracM3} but does not satisfy Dirac's condition.
Consider, for example, the graph $G_n$ on $10n+2$ vertices, $n\geq 2$,
which is defined as follows:
its vertex set is $\cup_{i=0}^5V_i$, where $V_0, V_1,\dots,V_5$ are
pairwise disjoint sets of cardinality $|V_0|=|V_5|=n$, $|V_1|=|V_4|=3n, |V_2|=|V_3|=n+1$
and two vertices of $G_n$ are adjacent if and only if they both belong to $V_i\cup V_{i+1}$
for some $i\in \{0,1,2,3,4\}$.
It is not difficult to see that $G_n$ satisfies the condition of \cref{prop:localdiracM3}.
The graph $G_2$ is given in \cref{fig:localdirac}.
\begin{figure}
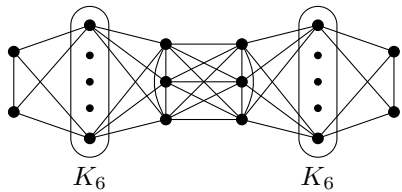

\begin{center}
\figlocaldirac
\end{center}
\caption{A graph satisfying the condition of \cref{prop:localdiracM3} but not Dirac's condition.}
\label{fig:localdirac}
\end{figure}

Thus in some cases the above method gives a larger class of Hamiltonian graphs than the original criterion gives.
In other cases the method does not generalize the original criterion,
but gives a local analogue which is a new sufficient condition for Hamiltonicity.
Consider, for example, the equivalent formulations of Ore's theorem given in \cref{ex:triviallocalore} and \cref{prop:localoreM3}.
By decreasing the radius of balls in the former from 2 to 1 we obtain the following result:

\begin{proposition}[\cite{hasratian90}]
\label{prop:localore1-balls}
A finite connected graph~$G$ on at least three vertices is Hamiltonian if every ball of radius~1 in~$G$ is a Ore graph.
\end{proposition}

Note that \cref{prop:localore1-balls} does not generalize Ore's theorem.
However, it describes a new class of Hamiltonian graphs.
In contrast with this, by decreasing the radius of the balls in \cref{prop:localoreM3}
from 3 to 2 we obtain the following result,
which is a generalization of Ore's theorem:
A finite connected graph $G$ on
at least three vertices is Hamiltonian if
$d(u)+d(v)\geq |M_2(x)|$
for every vertex $x\in V(G)$
and for all pairs of non-adjacent interior
vertices $u,v$ of the ball $G_2(x)$.

The conditions in this result can be relaxed slightly, giving another generalization of Ore's theorem:
\begin{proposition}[\cite{asratyan85}]
\label{prop:localoreM2}
Let $G$ be a finite connected graph on
at least three vertices such that
$d(u)+d(v)\geq |M_2(x)|$
for every path $uxv$ with $uv\notin E(G)$.
Then $G$ is Hamiltonian.
\end{proposition}

Finally, using the structure of balls, we can relax the condition in
\cref{prop:localoreM2}
and obtain a stronger result, \cref{oldthm:L0},
where $\card{M_2(x)}$ is replaced by $\card{N(u)\cup N(v)\cup N(x)}$.

Note that for regular graphs \cref{prop:localoreM2} can be reformulated as follows:
\begin{proposition}[\cite{asratian06}]
\label{prop:localoreM2oldregular}
A finite, connected, $k$-regular graph $G$, $k\ge2$,
is Hamiltonian if $2k\ge\card{M_2(x)}$ for every vertex $x\in V(G)$.
\end{proposition}
The condition $2k\geq |M_2(x)|$ in
\cref{prop:localoreM2oldregular}
can be rewritten as
\[2k\geq |M_2(x)|=1+k+|N_2(x)|,\]
which is equivalent to $|N_2(x)|<k$.
Therefore \cref{prop:localoreM2oldregular} can be reformulated as follows:

\begin{proposition}
\label{prop:localoreM2regular}
A finite, connected, $k$-regular graph $G$, $k\ge2$,
is Hamiltonian if the number of vertices at distance~$2$ from
any vertex in $G$ is less than~$k$.
\end{proposition}

\begin{remark}
\label{remark:local-but-not-global-ore-graphs}
Note that for any integer $n\geq 3$
there is a graph~$G$ of diameter~$n$ that satisfies the conditions of
\cref{prop:localore1-balls,prop:localoreM2,prop:localoreL0,prop:localoreM2oldregular,prop:localoreM2regular},
but is not an Ore-graph.
Consider, for example, the graph $G(p,2n)$, $n\geq 3$, which is defined as follows:
its vertex set is $V_1\cup\dotsb\cup V_{2n}$,
where $V_1,\dotsc,V_{2n}$ are pairwise disjoint sets of cardinality $p\geq 2$,
and two vertices of $G(p,2n)$ are adjacent if and only if
they both belong to $V_1\cup V_{2n}$ or to $V_i\cup V_{i+1}$
for some $i\in\{1,2,\dotsc,2n-1\}$.
Clearly, $G(p,2n)$ is a $(3p-1)$-regular graph of diameter~$n$.
It is not difficult to verify that the graph $G(p,2n)$ satisfies the conditions of \cref{prop:localore1-balls,prop:localoreM2,prop:localoreL0,prop:localoreM2oldregular,prop:localoreM2regular}.
\end{remark}


\section{Localization of Bondy's theorem}
\label{sec:bondy}

A subset of vertices in a graph $G$ is called \emph{independent} if no two of its elements are adjacent.
A cycle~$C$ in a graph~$G$ is called \emph{dominating} if $V(G)\setminus V(C)$ is an independent set.
Bondy obtained the following result:
\begin{theorem}[Bondy~\cite{bondy80}]
\label{oldthm:bondy}
Let $G$ be a 2-connected finite graph such that
$d(x)+d(y)+d(z)\ge \card{V(G)}+2$
for every triple of independent vertices $x,y,z$ of~$G$.
Then every longest cycle of $G$ is dominating.
\end{theorem}

We will find a local analogue of this theorem using our method of localization.
First we will show that Bondy's theorem is equivalent to the following:

\begin{proposition}
\label{prop:equivbondy}
Let $G$ be a connected finite graph on at least three vertices such that
every ball of radius~4 in~$G$ is 2-connected and
$d(x)+d(y)+d(z)\geq |M_4(v)|+2$
for every vertex $v\in V(G)$
and for all triples of
independent interior vertices $x, y, z$ of $G_4(v)$.
Then every longest cycle of $G$ is dominating.
\end{proposition}

\begin{proof}
Let $G$ satisfy the conditions of \cref{oldthm:bondy} and let $v\in V(G)$.
Then $d(x)+d(y)+d(z)\geq |M_4(v)|+2$
for all triples of independent interior vertices $x,y,z$ of the ball $G_4(v)$.
Now we show that $G_4(v)$ is $2$-connected.
It is evident if $G_4(v)=G$.
Assume that $G_4(v)\not=G$ and that $x$ is a cut vertex of $G_4(v)$.
Clearly $d(v,x)\le3$, and there is a neighbor~$u$ of~$x$
such that $d(v,u)=d(v,x)+1$ and every $(v,u)$-path in $G_4(v)$ contains~$x$.
Let $P$ be a shortest $\pathto{v}{u}$-path in $G-x$
(such a path exists since $G$ is $2$-connected).
Then $P$ contains at least one vertex from the set $N_5(v)$.
Let $y$ be the first one and $y'$ the last one of them.
Then $vPy=vv_1v_2v_3v_4y$ where $d(v,v_i)=i$ and $v_i\ne x$ for $i=1,2,3,4$.
Furthermore, let $z$ be the successor of $y'$ on $y'Pu$.
Clearly, $z\in N_4(v)$, and $v_3$~and $z$ are in different components of $G_4(v)-x$.
Then $d(v)+d(v_3)+d(z)\le\card{V(G)}-2$, a contradiction.
Thus every ball of radius~4 in~$G$ is 2-connected,
so the conditions of \cref{prop:equivbondy} hold.

Conversely, suppose that the conditions of \cref{prop:equivbondy} hold.
Consider a vertex $v\in V(G)$ satisfying
$ |M_4 (v)| = \max_{u\in V(G)}|M_4 (u)|$,
and suppose that $M_4(v)\not=V(G)$.
Then $N_5(v)\not=\emptyset$.
Let $vx_1x_2x_3x_4x_5$ be a path in $G$ where $d(v,x_i)=i$.
If there is a vertex $ y\in N_3(v)$ with $d(x_3,y)\geq 3$ then
\[d(v)\leq |M_4(v)|-|M_1(x_3)|-|M_1(y)|-1=|M_4(v)|-3-d(x_3)-d(y).\]
Thus $d(x_3)+d(y)+d(v)<|M_4(v)|$ for a triple of independent interior
vertices $x_3, y, v$ of the ball $G_4(v)$, which contradicts the condition of
\cref{prop:equivbondy}. Therefore $d(x_3,y)\leq 2$ for every $ y\in N_3(v)$.
Then it is easy to see that $M_4(x_3)$ contains
$\set{v}\cup N_1(v)\cup N_3(v)\cup N_4(v)\cup N_5(v)$.
We will show that $M_4(x_3)$ also contains $N_2(v)$.
Suppose to the contrary that there is $z\in N_2(v)\setminus M_4(x_3)$.
Then $d(z,x_1)=3$ since $z\notin M_4(x_3)$.
Also clearly $d(x_1,x_4)=3$, and $d(z,x_4)\ge3$ since $N(z)\cap N_3(v)=\emptyset$.
Then $d(z)+d(x_1)+d(x_4)\le M_4(x_1)-3$, a contradiction. Thus $N_2(v)\subset M_4(x_3)$
which implies $M_4(v)\cup N_5(v)\subseteq M_4(x_3)$.
Since
$N_5(v)\not=\emptyset$,
we have $|M_4(x_3)|>|M_4(v)|$, which contradicts our assumption.
This contradiction proves that $M_4(v)=V(G)$.
This implies that the conditions of \cref{oldthm:bondy} hold.
Thus \cref{prop:equivbondy} is equivalent to \cref{oldthm:bondy}.
\end{proof}

Now, according to step~2 of our method of localization,
we should try to show that the following
condition guarantees that every longest cycle of $G$ is dominating:
$d(x)+d(y)+d(z)\geq |M_3(v)|+2$
for every vertex $v\in V(G)$
and for all triples of
independent interior vertices $x, y, z$ of $G_3(v)$.
We will show that indeed a slightly stronger result holds:

\begin{theorem}
\label{thm:localbondy}
Let $G$ be a connected finite graph on at least three vertices such that
every ball of radius~3 in $G$ is 2-connected and
$d(x)+d(y)+d(z)\ge\card{M_3(v)}+2$
for every vertex $v\in V(G)$ and
for all triples of independent vertices $x,y,z$ of $G_2(v)$.
Then every longest cycle of $G$ is dominating.
\end{theorem}

We need the two lemmas below, which we will also use for the proof of a result on infinite graphs in \cref{sec:infinite}.
\begin{remark}
Note that the restrictions on the cycles~$C'$ in the lemmas are needed for the infinite case only;
to prove \cref{thm:localbondy} it is sufficient if $C'$ is any cycle longer than~$C$.
\end{remark}

\begin{lemma}
\label{lem:localbondylemma}
Let $G$ be a graph satisfying the conditions of \cref{thm:localbondy},
and let $C$ be a cycle of~$G$ such that
$G-C$ has some connected subgraph~$H$ containing at least two vertices.
Choose a direction $\dir C$ of~$C$.
If $uy$ and $vz$ are two non-adjacent edges in $G$ such that
$u,v\in V(H)$, $y,z\in V(C)$,
$d(y^+,z^+)\le2$,
and $N(u)\subseteq V(H)\cup V(C)$,
then there is a longer cycle $C'$
that differs from $C$ only in $H\cup G_7(z)$,
such that either $V(C)\subset V(C')$
or $\card{V(C)\setminus V(C')}=1$ and
the successor on~$C$ of the unique vertex in $V(C)\setminus V(C')$ is adjacent to~$z^+$.
\end{lemma}

We will call such a cycle $C'$ a \emph{desired cycle}.
Note that unless stated otherwise, any desired cycle that will be encountered will satisfy $V(C)\subset V(C')$.

\begin{proof}
Suppose that there is no desired cycle.
Let $\dir P$ be a $\pathto uv$-path in $H$ directed from $u$ to $v$.
Clearly $d(y^+,z^+)\ne1$, as otherwise there would be a desired cycle
$yu\dir Pvz\revdir Cy^+z^+\dir Cy$.

Thus $d(y^+,z^+)=2$.
Let $W=N(u)\cap V(C)$ and $U=V(H)$.
Clearly the set $W^+\cup\set{u,z^+}$ is independent and $N(x^+)\cap U=\emptyset$ for all $x\in W$,
because otherwise there would be a desired cycle~$C'$
(note that $W\cup W^+\subset M_7(z)$ since $W\cup W^+\subset M_4(y^+)$ and $d(y^+,z)\le3$).
%
Thus $N(u)\cap N(x^+)\subseteq W$ for all $x\in W$.
Similarly $N(z^+)\cap U=\emptyset$, so $N(u)\cap N(z^+)\subseteq W$.
%
Furthermore
\begin{equation}
d(x_1^+,x_2^+)=2 \text{ for all distinct } x_1,x_2\in W,
\end{equation}
because otherwise $d(x_1^+,x_2^+)\ge3$ for some $x_1,x_2\in W$,
so $N(x_1^+)\cap N(x_2^+)=\emptyset$ and therefore
$d(x_1^+)+d(x_2^+)+d(u)\le\card{M_3(u)}+\card{W}-\card{W^+\cup\set{u}}=\card{M_3(u)}-1$.

Pick $x\in W$ such that $N(u)\cap V(x^+\dir C z^-)=\emptyset$.
Then $x^+z^+\notin E(G)$.
If $d(x^+,z^+)\ge3$,
then $d(x^+)+d(z^+)+d(u)\le\card{M_3(y^+)}+\card{W}-\card{W^+\cup\set{u}}=\card{M_3(y^+)}-1$.
Thus $d(x^+,z^+)=2$.

We will partition $M_3(x^+)$ into three sets.
Let
	$V_1=V(G-C) \cap M_3(x^+)$,
	$V_2=V(x^+\dir C z)\cap M_3(x^+)$,
	$V_3=V(z^+\dir C x)\cap M_3(x^+)$.
For each vertex~$w$, we denote by $d_i(w)$ the number $\card{N(w)\cap V_i}$, $i=1,2,3$.
Note that $x^+,z^+,u\in M_2(x^+)$,
so $N(x^+)\cup N(z^+)\cup N(u)\subseteq M_3(x^+)=V_1\cup V_2\cup V_3$.
The \lcnamecref{lem:localbondylemma} now follows from
the three claims below, since they imply that
\begin{equation*}
d(x^+)+d(z^+)+d(u)=\sum_{i=1}^3\bigl(d_i(x^+)+d_i(z^+)+d_i(u)\bigr)\le\card{M_3(x^+)}+1,
\end{equation*}
a contradiction.

\begin{claim}
$d_1(x^+)+d_1(z^+)+d_1(u)\le\card{V_1}-1$.
\end{claim}
This follows from the fact that no two vertices of
$x^+$, $z^+$, and $u$ can have any common neighbors outside~$C$
(otherwise there would be a desired cycle),
and that $u\in V_1\setminus\bigl(N(x^+)\cup N(z^+)\cup N(u)\bigr)$.

\pagebreak[2]
\begin{claim}
$d_2(x^+)+d_2(z^+)+d_2(u)\le\card{V_2}+1$.
\end{claim}
By the choice of~$x$, $N(u)\cap V_2\subseteq\set z$,
that is, $d_2(u)\le1$.
Furthermore, if a vertex $b$ belongs to $N(x^+)\cap V_2$
then $b^-\notin N(z^+)$,
because otherwise there would be a desired cycle
$xu\dir Pvz\revdir Cbx^+\dir Cb^-z^+\dir Cx$.
Thus for any vertex $w\in N(x^+)\cap N(z^+)\cap V_2$ we can
“step backwards” through $w^-$, $w^{--}$ etc. to find a vertex
$q\in V_2\setminus\bigl(N(x^+)\cup N(z^+)\bigr)$ such that
$V(q^+\dir C w)\subset N(x^+)$ (possibly $q=x^+$).
This implies that $\card{N(x^+)\cap N(z^+)\cap V_2}\le\card{V_2\setminus\bigl(N(x^+)\cup N(z^+)\bigr)}$,
so
\begin{equation}
\label{lem:localbondylemma:eq:d2}
\begin{aligned}
d_2(x^+)+d_2(z^+)
&=\card[\big]{\bigl(N(x^+)\cup N(z^+)\bigr)\cap V_2}+\card{N(x^+)\cap N(z^+)\cap V_2}\\
&\le\card[\big]{\bigl(N(x^+)\cup N(z^+)\bigr)\cap V_2}
	+\card[\big]{V_2\setminus\bigl(N(x^+)\cup N(z^+)\bigr)}\\
&=\card{V_2}.
\end{aligned}
\end{equation}
Therefore $d_2(x^+)+d_2(z^+)+d_2(u)\le\card{V_2}+1$.

\begin{claim}
\label{lem:localbondylemma:claim3}
$d_3(x^+)+d_3(z^+)+d_3(u)\le\card{V_3}+1$.
\end{claim}
If $z\in W$, then $d_3(u)=\card{W}-1$, otherwise $d_3(u)=\card{W}$.
We will show that $d_3(x^+)+d_3(z^+)$ is at most $\card{V_3}-\card{W}+2$
and $\card{V_3}-\card{W}+1$, respectively, in these cases.

For any vertex $b\in N(x^+)\cap V_3$
it holds that $b^+\notin N(z^+)$ and $b^{++}\notin N(z^+)$,
as otherwise there would be a desired cycle
$xu\dir P vz\revdir Cx^+b\revdir Cz^+g\dir Cx$,
where $g\in\set{b^+,b^{++}}$.
(Note that if $g=b^{++}$ then $V(C)\setminus V(C')=\set{b^+}$
and the vertex $b^{++}$, the successor of $b^{+}$, is adjacent to~$z^+$.
If $g=b^+$ then $V(C)\subset V(C')$.)
Also at most one of $b^+$ and $b^{++}$ are in $W^+$.
This means that there are two possibilities:
\begin{itemize}
\item both $b^+$ and $b^{++}$ lie in $N(x^+)\cup W^+$,
in which case at least one of them is adjacent to $x^+$, or
\item at least one of $b^+$ and $b^{++}$ does not lie in $N(x^+)\cup W^+$.
\end{itemize}
Thus for any vertex $w\in N(x^+)\cap N(z^+)\cap V_3$ we can
“step forwards” through $w^+$, $w^{++}$ etc,
to find a vertex $q\in V_3\setminus\bigl(N(x^+)\cup N(z^+)\cup W^+\bigr)$
such that $V(w\dir C q^-)\subset N(x^+)\cup W^+$,
unless $V(w\dir C x)\subset N(x^+)\cup W^+$.
That is, for each $w\in N(x^+)\cap N(z^+)\cap V_3$ except possibly “the last” one along $\dir C$,
we can find a unique $q\in V_3\setminus W^+$ that is not adjacent to $x^+$ or $z^+$.
Thus

\begin{equation}
\begin{aligned}
\card{N(x^+)\cap N(z^+)\cap V_3}
&\le\card{V_3\setminus\bigl(N(x^+)\cup N(z^+)\cup W^+\bigr)}+1\\
&=\card{V_3\setminus\bigl(N(x^+)\cup N(z^+)\bigr)}-\card{W^+\cap V_3}+1.
\end{aligned}
\end{equation}
If $z\notin W$, we furthermore get
\begin{equation}
\card{N(x^+)\cap N(z^+)\cap V_3}
\le\card[\big]{V_3\setminus\bigl(N(x^+)\cup N(z^+)\bigr)}-\card{(W^+\cap V_3)\cup\set{z^+}}+1.
\end{equation}

Thus if $z\in W$, then with calculations similar to \cref{lem:localbondylemma:eq:d2} we get,
using $W^+\cap V_3=W^+\setminus\set{x^+}$,
that $d_3(x^+)+d_3(z^+)\le \card{V_3}-\card{W^+}+2$.
If $z\notin W$, we get $d_3(x^+)+d_3(z^+)\le \card{V_3}-\card{W^+}+1$.
In both cases $d_3(x^+)+d_3(z^+)+d_3(u)\le\card{V_3}+1$.
\end{proof}

\begin{lemma}
\label{lem:localbondythm}
Let $G$ be a graph satisfying the conditions of \cref{thm:localbondy}.
Let $C$ be a cycle of~$G$ and $u$ a vertex in $G-C$ with a neighbor on~$C$.
Then if $N(u)\nsubseteq V(C)$, there is a longer cycle $C'$
that contains a neighbor of~$u$,%
\footnote{Note that the vertex~$u$ itself can be on~$C'$ or outside of~$C'$}
and differs from~$C$ only in~$G_{12}(u)$.
\end{lemma}

We will call such a cycle $C'$ a \emph{feasible cycle}.
Note that unless stated otherwise,
any feasible cycle that will be encountered will satisfy $V(C)\subset V(C')$;
this clearly implies that it contains a neighbor of~$u$.

\begin{proof}
Suppose that $N(u)\nsubseteq V(C)$ and that
there is no feasible cycle.
Let $H$ be the component of $G_7(u)-C$ containing~$u$
and let $U=V(H)$.
Then $N(u)\subseteq U\cup V(C)$.
Choose a direction $\dir C$ of~$C$.
Let $W_u=N(u)\cap V(C)$.
Clearly the set $W_u^+$ is independent and $N(w^+)\cap U=\emptyset$ for all $w\in W_u$,
since otherwise there would be a feasible cycle.

We shall find non-adjacent edges $uy$ and $vz$ such that
$u,v\in U$, $y,z\in V(C)$, and $z\in M_3(y^+)$ as follows:
\begin{itemize}
\item
If $\card{W_u}=1$, let $y$ be the unique vertex in~$W_u$.
This gives us the edge~$uy$.
Since $G_3(y^+)$ is 2-connected,
there is a $\pathto{u}{y^+}$-path~$P$ in $G_3(y^+)-y$,
and because $U$~is the vertex set of a component of $G_7(u)-C$,
there is a vertex~$z\in V(P)\cap V(C)$ such that
all other vertices of $u\dir P z$ are in~$U$.
Letting $v=z^-$ (with respect to $\dir P$),
we obtain the required edge $vz$.
\item
If $\card{W_u}\ge2$ but $N(w)\cap U=\set{u}$ for each vertex~$w\in W_u$,
then let $y$ be any vertex in $W_u$,
giving us the edge~$uy$.
Now we can, as above, find the edge $vz$ on
a path from any vertex in~$U\setminus\set{u}$ to $y^+$ in $G_3(y^+)-u$.
\item
Finally, if $\card{W_u}\ge2$ and $W_u$ contains vertices with neighbors in~$U\setminus\set{u}$,
we get the edges $uy$ and $vz$ by picking
$y$ and $z$ in~$W_u$ such that $z$ has a neighbor~$v$ in~$U\setminus\set{u}$.
\end{itemize}
Note that by the construction above,
\begin{equation}
\label{lem:localbondythm:eq:y+v-close-or-uz-adj}
\text{either $v\in M_3(y^+)$ or $u$ and $z$ are adjacent.}
\end{equation}
Thus $v\in N_5(u)$, which means that $N(v)\subseteq U\cup V(C)$.

Clearly $z\notin W_u^+$.
Also the set $W_u^+\cup\set{z^+}$ is independent and $N(z^+)\cap U=\emptyset$,
since otherwise there would be a feasible cycle.
If $d(y^+,z^+)=2$ then, by \cref{lem:localbondylemma},
there is a longer cycle~$C'$ where either $V(C)\subseteq V(C')$
or $V(C)\setminus V(C')=\set{w}$ for some $w$ with $w^+z^+\in E(G)$.
In both cases the vertex~$u$ has a neighbor on~$C'$
(in the second case this holds because $y^+z^+\notin E(G)$ so $y\ne w$, and hence $y\in V(C')$).
Furthermore $C'$ differs from $C$ only in $H\cup G_7(z)\subseteq G_{12}(u)$,
since $d(u,z)\le5$,
so it is a feasible cycle.
We can thus conclude that
\begin{equation}
\label{thm:localbondy:eq:d(y+-z+)>=3}
d(y^+,z^+)\ge3.
\end{equation}

Define $W_v=N(v)\cap V(C)$.
Clearly the set $W_u^+\cup W_v^+$ is independent
and $N(w^+)\cap U=\emptyset$ for all $w\in W_v$,
since otherwise there would be a feasible cycle.
It is also easy to see that
$N(y^+)\cap N(v)\subseteq W_v$
and $N(v)\cap N(w^+)\subseteq W_v$ for all $w\in W_v$,
as $N(v)\setminus V(C)\subseteq U$.
Now, if $u$ and $z$ are adjacent, then
$d(y^+)+d(z^+)+d(u)\le\card{M_3(u)}+\card{W_u}-\card{W_u^+\cup\set{u}}=\card{M_3(u)}-1$,
a contradiction.
Thus $v\in M_3(y^+)$ by \cref{lem:localbondythm:eq:y+v-close-or-uz-adj}.
Since, by~\cref{thm:localbondy:eq:d(y+-z+)>=3}, $N(y^+)\cap N(z^+)=\emptyset$, we obtain
\begin{equation}
\label{thm:localbondy:eq:d(y+-v)=3}
d(y^+,v)=3,
\end{equation}
because if $y^+\in M_2(v)$ then
$d(y^+)+d(z^+)+d(v)\le\card{M_3(v)}+\card{W_v}-\card{W_v^+\cup\set{v}}=\card{M_3(v)}-1$.
Also, since $W_v^+\subset M_3(z)$,
\begin{equation}
\label{thm:localbondy:eq:d(y+-z)=3}
d(y^+,z)=3,
\end{equation}
as if $y^+\in M_2(z)$ then
$d(y^+)+d(z^+)+d(v)\le\card{M_3(z)}+\card{W_v}-\card{W_v^+\cup\set{v}}=\card{M_3(z)}-1$.

By \cref{thm:localbondy:eq:d(y+-v)=3},
there is a vertex
\begin{equation}
\label{thm:localbondy:eq:defofs}
s\in N(v)\cap N_2(y^+).
\end{equation}
Clearly $s\in U$ or $s\in V(C)$.
We will show that $s\in U$.
Assume that $s\in V(C)$,
which means that $s\in W_v$ and $s^+\in W_v^+$.
Then $W_v^+\subset M_3(s)$ and $s^+y^+\notin E(G)$.
Thus $d(y^+,s^+)=2$, since if $d(y^+,s^+)\ge3$ then
$d(y^+)+d(s^+)+d(v)\le\card{M_3(s)}+\card{W_v}-\card{W_v^+\cup\set{v}}=\card{M_3(s)}-1$.
Also $d(s^+,z^+)=2$,
since if $d(s^+,z^+)\ge3$ then
$d(s^+)+d(z^+)+d(v)\le\card{M_3(v)}+\card{W_v}-\card{W_v^+\cup\set{v}}=\card{M_3(v)}-1$.
Now we have $z^+,y^+,v\in M_2(s^+)$,
$d(y^+,v)=3$ by~\cref{thm:localbondy:eq:d(y+-v)=3}, and
$d(y^+,z^+)\ge3$ by~\cref{thm:localbondy:eq:d(y+-z+)>=3}.
Then
$d(y^+)+d(z^+)+d(v)\le\card{M_3(s^+)}+\card{W_v}-\card{W_v^+\cup\set{v}}=\card{M_3(s^+)}-1$,
a contradiction.
Therefore we can conclude that $s\in U$.

Define $W_{s}=N(s)\cap V(C)$.
Clearly the set $W_u^+\cup W_v^+\cup W_{s}^+$ is independent
and $N(w^+)\cap U=\emptyset$ for all $w\in W_{s}$,
since otherwise there would be a feasible cycle.
Also $N(s)\subseteq U\cup V(C)$, because $d(u,s)\le4$.
Thus $N(w^+)\cap N(s)\subset W_s$ for all $w\in W_u\cup W_v\cup W_{s}$.
Now if $z^+\in M_2(s)$ then, since $y^+\in M_2(s)$ by \cref{thm:localbondy:eq:defofs}
and $d(y^+,z^+)\ge3$ by~\cref{thm:localbondy:eq:d(y+-z+)>=3},
$d(y^+)+d(z^+)+d(s)\le\card{M_3(s)}+\card{W_s}-\card{W_s^+\cup\set{s}}=\card{M_3(s)}-1$.
Thus
\begin{equation}
\label{thm:localbondy:eq:d(z+-s)=3}
d(z^+,s)=3.
\end{equation}
This and \cref{thm:localbondy:eq:defofs} imply that
\begin{equation}
\label{thm:localbondy:eq:d(z-s)=2}
d(z,s)=2.
\end{equation}
Furthermore,
\begin{equation}
\label{thm:localbondy:eq:d(y+-ws+)=3}
d(y^+,w^+)=2 \text{ for all } w\in W_{s}\setminus\set y,
\end{equation}
since if $d(y^+,w^+)\ge3$ for some $w\in W_{s}\setminus\set y$ then
$d(y^+)+d(w^+)+d(s)\le\card{M_3(s)}+\card{W_{s}}-\card{W_{s}^+\cup\set{s}}=\card{M_3(s)}-1$,
a contradiction.
Using the same argument, we get
\begin{equation}
\label{thm:localbondy:eq:d(ws1+-ws2+)=3}
d(w_1^+,w_2^+)=2 \text{ for all distinct } w_1,w_2\in W_{s}.
\end{equation}
This means that
\begin{equation}
\label{thm:localbondy:eq:zWs+intersection}
N(z)\cap W_{s}^+=\emptyset,
\end{equation}
since otherwise $zw^+\in E(G)$ for some $w\in W_{s}$,
so, by \cref{thm:localbondy:eq:d(y+-ws+)=3},
$d(y^+)+d(z^+)+d(s)\le\card{M_3(w^+)}+\card{W_{s}}-\card{W_{s}^+\cup\set{s}}=\card{M_3(w^+)}-1$.

We will now show that $\card{W_v}\ge2$.
Suppose that $\card{W_v}=1$, that is, $W_v=\set z$.
By \cref{thm:localbondy:eq:d(y+-z)=3} there is some $w\in N(z)\cap N_2(y^+)$.
But then this, \cref{thm:localbondy:eq:d(y+-z+)>=3}, and \cref{thm:localbondy:eq:d(y+-v)=3} imply
\begin{equation*}
d(y^+)+d(z^+)+d(v)\le\card{M_3(w)}+\card{W_v}-\card{\set{y^+,z^+,v}}=\card{M_3(w)}-2.
\end{equation*}
Thus $\card{W_v}\ge2$.
Also, $\card{N(z)\cap U}\ge2$, since otherwise
$N(z)\cap U=\set v$ and
$N(z)\cap N(s)\subseteq W_s\cup\set v$,
so by
\cref{thm:localbondy:eq:d(y+-z)=3,thm:localbondy:eq:defofs,thm:localbondy:eq:d(z-s)=2,thm:localbondy:eq:zWs+intersection}
we would have
$d(y^+)+d(z)+d(s)\le\card{M_3(s)}+\card{W_{s}\cup\set{v}}-\card{W_{s}^+\cup\set{z,s}}
	=\card{M_3(s)}-1$,
a contradiction.

Now $\card{W_v}\ge2$ and $\card{N(z)\cap U}\ge2$
imply that there exists vertices $y'\in W_v\setminus\set z$
and $v'\in \bigl(N(z)\cap U\bigr)\setminus\set v$.
For simplicity we shall also use the alias $u'=v$.
Then $d(y'^+,z^+)\le2$, as otherwise
$d(y'^+)+d(z^+)+d(v)\le\card{M_3(v)}+\card{W_v}-\card{W_v^+\cup\set{v}}=\card{M_3(v)}-1$,
a contradiction.
Finally, by applying \cref{lem:localbondylemma}
with $u'y'$ and $v'z$ instead of $uy$ and $vz$,
and with the same subgraph~$H$ and cycle~$C$,
we get, as earlier, a longer cycle~$C'$ where either $V(C)\subseteq V(C')$
or $V(C)\setminus V(C')=\set{w}$ for some $w$ with $w^+z^+\in E(G)$.
Again, in both cases the vertex~$u$ has a neighbor on~$C'$
(in the second case this holds because $y^+z^+\notin E(G)$ so $y\ne w$, and hence $y\in V(C')$),
and $C'$ differs from $C$ only in $H\cup G_7(z)\subseteq G_{12}(u)$,
so it is a feasible cycle.
The \lcnamecref{lem:localbondythm} follows.
\end{proof}

\begin{boldproof}[Proof of \cref{thm:localbondy}]
Let $C$ be a longest cycle in a graph $G$ satisfying the condition of \cref{thm:localbondy}.
Then, by \cref{lem:localbondythm}, $N(u)\subseteq V(C)$ for every vertex $u\in V(G)\setminus V(C)$. Therefore $C$ is a dominating cycle.
\end{boldproof}

Note that the diameter of a graph satisfying the conditions of \cref{oldthm:bondy}
does not exceed~5.
In contrast with this,
for any integers $n\geq 4$ and $p\geq 3$ there is a $(3p-1)$-regular graph
of diameter~$n$ satisfying the conditions of \cref{thm:localbondy}
(see, for example, the graph $G(p,2n)$ in \cref{remark:local-but-not-global-ore-graphs}).

In the next section we will use \cref{thm:localbondy} to find
local analogues of two well-known results.

Note that \Cref{thm:localbondy} is not a generalization of \cref{oldthm:bondy},
as there are 2-connected graphs satisfying the conditions of \cref{oldthm:bondy}
where some balls of radius~3 are not 2-connected.
An example of such a graph can be seen in \cref{fig:threeballnottwoconnected}.
However, by relaxing the requirement
that all balls of radius~3 are 2-connected,
one can obtain the following \lcnamecref{thm:localbondygeneral}
which is a generalization of \cref{oldthm:bondy,thm:localbondy}:

\begin{figure}
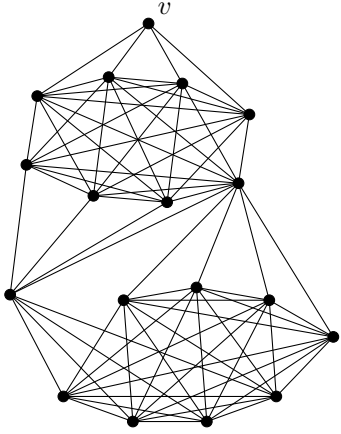

\begin{center}
\figthreeballnottwoconnected
\end{center}
\caption{A 2-connected graph~$G$ where the ball $G_3(v)$ is not 2-connected.}
\label{fig:threeballnottwoconnected}
\end{figure}

\begin{theorem}[\cite{granholm-dominatingcurve}]
\label{thm:localbondygeneral}
Let $G$ be a connected finite graph on at least three vertices such that
$d(x)+d(y)+d(z)\ge\card{M_3(v)}+2$
for every vertex $v\in V(G)$ and
for all triples of independent vertices $x,y,z$ of $G_2(v)$.
Furthermore assume that for every ball $G_3(v)$ of radius~3 that is not 2-connected,
and for every cut vertex $u$ in $G_3(v)$,
the following property holds:
$u\in N_2(v)$ and
for all $a,b\in N_3(v)$ where $a$ and $b$ are in different components of $G_3(v)-u$
there is a vertex $w\in N_4(v)$ such that $aw,bw\in E(G)$.
Then every longest cycle of $G$ is dominating.
\end{theorem}


\section{Local analogues of two well-known theorems}
\label{sec:kappa}

H\"aggkvist and Nicoghossian proved the following theorem:

\begin{theorem}[H\"aggkvist and Nicoghossian~\cite{haggkvist81}]
\label{oldthm:haggkvist}
Let $G$ be a 2-connected finite graph
such that $d(u)\geq \tfrac13\bigl(|V(G)|+\kappa(G)\bigr)$ for every vertex~$u\in V(G)$.
Then $G$ is Hamiltonian.
\end{theorem}

Later Bauer, Broersma, Veldman, and Rao
generalized their result as follows:

\begin{theorem}[Bauer et~al.~\cite{bauer89}]
\label{oldthm:kappa}
Let $G$ be a 2-connected finite graph
such that $d(x)+d(y)+d(z)\geq |V(G)|+\kappa(G)$ for every triple of
independent vertices $x, y, z$ of $G$.
Then $G$ is Hamiltonian.
\end{theorem}
A short proof of \cref{oldthm:kappa} was obtained by Wei~\cite{wei99}.
We will provide local analogues of \cref{oldthm:haggkvist} and \cref{oldthm:kappa}.

In section 4 we proved that \cref{prop:equivbondy} is equivalent to \cref{oldthm:bondy}.
Using a similar argument one can prove that
\cref{oldthm:kappa} is equivalent to the following proposition:
A connected finite graph $G$ on at least three vertices is Hamiltonian
if every ball of radius~4 in~$G$ is $2$-connected
and $d(x)+d(y)+d(z)\geq\card{M_4(v)}+\kappa\bigl(G_4(v)\bigr)$
for every vertex $v\in V(G)$
and for all triples of independent interior vertices $x,y,z$ of $G_4(v)$.
(For more details, see \cite[Prop.~3.3]{asratian07} where a similar result is proved.)

We will prove the following local analogue of \cref{oldthm:kappa},
using an argument similar to the one in~\cite{wei99}:

\begin{theorem}
\label{thm:localkappa}
Let $G$ be a connected finite graph on at least three vertices
such that every ball of radius~$3$ in~$G$ is 2-connected and
$d(x)+d(y)+d(z)\ge\card{M_3(v)}+\kappa\bigl(G_3(v)\bigr)$
for every vertex $v\in V(G)$
and for all triples of independent interior vertices $x,y,z$ of $G_3(v)$.
Then $G$ is Hamiltonian.
\end{theorem}

\begin{proof}
Let $G$ be a graph satisfying the conditions of \cref{thm:localkappa},
and let $C$ be a longest cycle of~$G$.
Assume that $C$ is not a Hamilton cycle.
By \cref{thm:localbondy}, $C$~is dominating.
Choose a direction $\dir C$ of~$C$.
Pick $v\in V(G-C)$.
Since $C$ is dominating, $N(v)\subseteq V(C)$.
Let $W=N(v)$ and let $x_1,\dotsc,x_t$ be the vertices of $W$ in order around~$C$.
Let $C_i=x_i^+\dir C x_{i+1}$ for $i=1,\dotsc,t-1$ and $C_t=x_t^+\dir C x_1$.
For easier notation, let $\kappa_3(w)$ denote $\kappa\bigl(G_3(w)\bigr)$
for any vertex~$w$.
Also, we will let $M_3\open(w)$ denote the set of interior vertices of $G_3(w)$.

Since $C$ is a longest cycle,
the set $W^+\cup\set v$ is independent
and $N(y^+)\cap N(z^+)\subseteq V(C)$ for any vertices $y,z\in W$.
Furthermore
\begin{equation}
\label{thm:localkappa:eq:d(w1+-w2+)=2}
d(y^+,z^+)=2 \text{ for all distinct } y,z\in W,
\end{equation}
since if $d(y^+,z^+)\ge3$ for some $y,z\in W$ then
$d(y^+)+d(z^+)+d(v)\le\card{M_3(v)}+\card{W}-\card{W^+\cup\set{v}}=\card{M_3(v)}-1$.
Also
\begin{equation}
\label{thm:localkappa:eq:Winternal}
x\in M_3\open(y^+) \text{ for all $x,y\in W$.}
\end{equation}
since otherwise for some $x,y\in W$ we can pick $w\in N(x)\setminus M_3(y^+)$,
and using \cref{thm:localkappa:eq:d(w1+-w2+)=2} we obtain that
$d(w,z^+)\ge2$ for every $z\in W$ and therefore
$d(y^+)+d(v)+d(w)\le\card{M_3(x^+)}+\card{W}-\card{W^+\cup\set{v,w}}=\card{M_3(x^+)}-2$.

Let $y$ and $z$ be arbitrary vertices in $W$.
Since $y^+$ and $z^+$ have no common neighbors outside of $C$
and furthermore are not adjacent to $v$, we get
\begin{equation}
\label{thm:kappa:eq:orig(3)}
\card{N(y^+)\setminus V(C)}+\card{N(z^+)\setminus V(C)}
\le\card{M_3(y^+)\setminus V(C)}-1.
\end{equation}
We will show now that for each $i=1,\dotsc,t$ and for every $s\in V(C_i)$
the following inequality holds:
\begin{equation}
\label{thm:kappa:eq:orig(1)help}
\card{N(y^+)\cap N(z^+)\cap V(x_i^+\dir C s)}
\le \card[\big]{M_3(y^+)\cap V(x_i^+\dir C s)
	\setminus\bigl(N(y^+)\cup N(z^+)\bigr)},
\end{equation}
by pairing each vertex~$w$ on $x_i^+\dir Cs$ that is adjacent to both $y^+$ and $z^+$
with a vertex~$q\in M_3(y^+)$ that is adjacent to neither.
We consider only the case when
$N(y^+)\cap N(z^+)\cap V(x_i^+\dir C s)\ne\emptyset$,
because otherwise \cref{thm:kappa:eq:orig(1)help} is evident.

Clearly, $x_i^+\dir C s$ is a part of either $y^+\dir Cz$ or $z^+\dir Cy$.
Assume first that it is a part of $y^+\dir Cz$.
Then for each $b\in N(y^+)\cap V(x_i^+\dir Cs)$ we have $b^-\notin N(z^+)$
since otherwise there would be a longer cycle.
Now, for each $w\in N(y^+)\cap N(z^+)$ that lies on $x_i^+\dir Cs$
we “step backwards” through $w^-,w^{--}$ etc.
to find the first vertex $q\notin N(y^+)$
and pair it with~$w$.
Clearly, $V(q^+\dir C w^-)\subset N(y^+)\setminus N(z^+)$,
which means that $q\notin N(z^+)$
and that no other vertex $w'\in N(y^+)\cap N(z^+)\cap V(x_i^+\dir Cs)$ can be paired with~$q$.
Also, $q\in V(x_i^+\dir Cs)$, because $x_i^+\notin N(y^+)$.
Clearly this pairing defines an injection from
the set $N(y^+)\cap N(z^+)\cap V(x_i^+\dir C s)$
to the set $M_3(y^+)\cap V(x_i^+\dir C s) \setminus\bigl(N(y^+)\cup N(z^+)\bigr)$.
Therefore \cref{thm:kappa:eq:orig(1)help} holds in this case.

Assume now that $x_i^+\dir Cs$ is a part of $z^+\dir Cy$.
Clearly, for each $b\in N(y^+)\cap V(z^+\dir Cy)$ we have $b^+\notin N(z^+)$,
since otherwise there would be a longer cycle.
Thus for any vertex $w\in N(y^+)\cap N(z^+)$ that lies on $x_i^+\dir Cs$
we can “step forwards” through $w^+,w^{++}$ etc.
to find the first vertex $q\notin N(y^+)$ and pair it with~$w$.
Let $w_1,\dotsc,w_p$ be the vertices of $N(y^+)\cap N(z^+)\cap V(x_i^+\dir Cs)$
in order along $\dir C$,
and let $q_1,\dotsc,q_p$ be the vertices paired with $w_1,\dotsc,w_p$, respectively.
If $p\ge2$ then $q_k$ lies on $w_k\dir Cw_{k+1}$ and, therefore, on $x_i^+\dir Cs$, for $k=1,\dotsc,p-1$.
However,
$q_p$ can be outside of $x_i^+\dir Cs$.
(This happens only if $y^+$ is adjacent to every vertex on $w_p\dir C s$.)
Note that $x_i^+\notin N(y^+)\cup N(z^+)$,
but, since we are stepping forwards,
$x_i^+$ will not be paired with any vertex~$w_k$, $1\leq k\leq p$.
Therefore, if $q_p\notin V(x_i^+\dir Cs)$ we can pair $w_p$ with $x_i^+$, that is, put $q_p=x_i^+$.
Now $q_1,\dotsc,q_p\in V(x_i^+\dir Cs)$
and our pairing process defines an injection from
the set $N(y^+)\cap N(z^+)\cap V(x_i^+\dir C s)$
to the set $M_3(y^+)\cap V(x_i^+\dir C s) \setminus\bigl(N(y^+)\cup N(z^+)\bigr)$,
which means that \cref{thm:kappa:eq:orig(1)help} holds in this case as well.

By adding $\card[\big]{\bigl(N(y^+)\cup N(z^+)\bigr)\cap V(x_i^+\dir Cs)}$
to both sides of \cref{thm:kappa:eq:orig(1)help} we get the following inequality:
\begin{equation}
\label{thm:kappa:eq:orig(1)}
\card{N(y^+)\cap V(x_i^+\dir Cs)}+\card{N(z^+)\cap V(x_i^+\dir Cs)}
	\le\card{M_3(y^+)\cap V(x_i^+\dir Cs)}.
\end{equation}
Finally, by letting $s=x_{i+1}$ (or $s=x_1$ if $i=t$) and summing over $i$ we obtain
\begin{equation}
\label{thm:kappa:eq:orig(2)}
\card{N(y^+)\cap V(C)}+\card{N(z^+)\cap V(C)}\le\card{M_3(y^+)\cap V(C)}.
\end{equation}

Pick $y_0\in W$ such that
\begin{equation}
\kappa_3(y_0^+)\le \kappa_3(z^+) \text{ for all } z\in W.
\end{equation}
For brevity we will use the following notation:
\begin{equation}
\begin{aligned}
\kappa_3=\kappa_3(y_0^+),\ 
&G_3=G_3(y_0^+),\\
M_3=M_3(y_0^+),\ 
&M_3\open=M_3\open(y_0^+).
\end{aligned}
\end{equation}
We shall prove that $\card W\ge\kappa_3+1$.
Clearly $\card W=\card{N(v)}\ge\kappa_3$ since $v\in M_3$.
Assume that $\card W=\kappa_3$, that is, $d(v)=\kappa_3$.
This means that $\card{W}\ge2$,
since $G_3(y_0^+)$ is 2-connected, that is, $\kappa_3\ge2$.
Then, choosing two vertices $y,z\in W$ we obtain,
by \cref{thm:kappa:eq:orig(2),thm:kappa:eq:orig(3)}, that
$d(y^+)+d(z^+)+d(v)\le\card{M_3(y^+)}-1+d(v)
	=\card{M_3(y^+)}+\kappa_3-1\le\card{M_3(y^+)}+\kappa_3(y^+)-1$,
a contradiction.
Thus $\card W\ge\kappa_3+1$.

Let $V_0$ be a vertex cut set of $G_3$ with $\card{V_0}=\kappa_3$,
and let $V_1,\dotsc,V_p$ be the vertex sets of the components of $G_3-V_0$.
Set $U_i=V_i\cap M_3\open$ and $Y_i=V_i\cap W^+$, $i=0,\dotsc,p$.
By \cref{thm:localkappa:eq:d(w1+-w2+)=2},
$W^+\subset M_3\open$.
Therefore $Y_i\subseteq U_i\subseteq V_i$.
It is clear that $U_0\ne\emptyset$.
Since $\card{V_0}=\kappa_3<\card{W^+}$,
we may assume that $Y_1\ne\emptyset$, so $U_1\ne\emptyset$.
Furthermore assume that the sets $V_i$ are ordered such that
the nonempty $U_i$ have the lowest indices,
that is, if $U_i=\emptyset$ then $U_j=\emptyset$ for all $j>i$.

\parenum
Suppose that there are at least four nonempty sets~$U_i$,
that is, $U_i\ne\emptyset$ for $i=0,1,2,3$.
For each $i=1,2,3$, pick $u_i\in U_i$ such that $u_i\in Y_i$ if $Y_i\ne\emptyset$.
Then $N(u_i)\cap Y_i=\emptyset$
and furthermore $N(u_1)\cap Y_0=\emptyset$ since $Y_1\ne\emptyset$,
so we obtain
\begin{multline}
\sum_{i=1}^3 d(u_i)
\le\sum_{i=1}^3\card{V_i\setminus Y_i}+3\card{V_0}-\card{Y_0}
\le\card{M_3\setminus W^+}+2\card{V_0}\\
=\card{M_3\setminus W^+}+2\kappa_3
\le\card{M_3}+\kappa_3-1,
\end{multline}
a contradiction.
Thus $U_3=\emptyset$.

\parenum
Suppose that there are three nonempty sets~$U_i$,
that is, $U_0$, $U_1$, and~$U_2$ are nonempty.
We consider two cases.
\begin{case}
$U_2\nsubseteq W\cup\set v$.
\end{case}
Pick $u_1\in Y_1$,
and pick $u_2\in Y_2$ if $Y_2\ne\emptyset$,
otherwise pick $u_2\in U_1\setminus(W\cup\set{v})$.
Again $N(u_i)\cap Y_i=\emptyset$
and $N(u_1)\cap Y_0=\emptyset$,
and furthermore $v\notin N(u_1)\cup N(u_2)$.
We obtain
\begin{multline}
d(u_1)+d(u_2)
\le\sum_{i=1}^2\card{V_i\setminus Y_i}+2\card{V_0}-\card{Y_0\cup\set v}
\le\card{M_3\setminus W^+}+\card{V_0}-1\\
=\card{M_3\setminus W^+}+\kappa_3-1
\le\card{M_3}+\kappa_3-\card{W^+}-1.
\end{multline}
Thus $d(u_1)+d(u_2)+d(v)\le\card{M_3}+\kappa_3-1$, a contradiction.
\begin{case}
\label{thm:kappa:case:U2inWandv}
$U_2\subseteq W\cup\set v$.
\end{case}
In this case $Y_2=\emptyset$,
so $v\notin V_2$ as otherwise
$(W\setminus V_2)\cup\defset{x^+}{x\in W\cap V_2}\subseteq V_0$,
implying that $\card{V_0}\ge\card{W}\ge\kappa_3+1$,
a contradiction
(note that $W\cap W^+=\emptyset$, so the union is disjoint).
We check two subcases:

\begin{subcase}
$U_2=V_2$.
\end{subcase}
In this case $V_2\subseteq W=N(v)$ and $v\in V_0$,
so $\card{V_0\setminus\set v}=\kappa_3-1\le\card{W^+}-2$ and thus $\card{Y_1}\ge2$.
Let $y^+,z^+\in Y_1$.
For each $x_i\in V_2$,
using \cref{thm:kappa:eq:orig(1)} with $s=x_i^-$ we obtain
\begin{equation}
\label{thm:kappa:eq:orig(1)V2}
\card{N(y^+)\cap V(x_{i-1}^+\dir Cx_i^-)}+\card{N(z^+)\cap V(x_{i-1}^+\dir Cx_i^-)}
	\le\card{M_3(y^+)\cap V(C_{i-1})}-1.
\end{equation}
Similarly for $x_i\in W\setminus V_2$,
using \cref{thm:kappa:eq:orig(1)} with $s=x_i$ we obtain
\begin{equation}
\label{thm:kappa:eq:orig(1)notV2}
\card{N(y^+)\cap V(x_{i-1}^+\dir Cx_i)}+\card{N(z^+)\cap V(x_{i-1}^+\dir Cx_i)}
	\le\card{M_3(y^+)\cap V(C_{i-1})}.
\end{equation}
Thus, by \cref{thm:kappa:eq:orig(3),thm:kappa:eq:orig(1)V2,thm:kappa:eq:orig(1)notV2},
$d(y^+)+d(z^+)\le\card{M_3(y^+)}-\card{V_2}-1$.
Furthermore $d(x)\le\card{V_2}-1+\card{V_0}$ for any $x\in V_2$
since $x\in M_3\open$.
Thus
$d(y^+)+d(z^+)+d(x)
\le\card{M_3(y^+)}+\kappa_3-2
\le\card{M_3(y^+)}+\kappa_3(y^+)-2$,
a contradiction.

\begin{subcase}
$V_2\setminus U_2\ne\emptyset$.
\end{subcase}
Pick $u_1\in Y_1$,
and furthermore pick two vertices $x\in U_2$ and $u_2\in V_2\setminus U_2$
such that $u_2\in N(x)$.
Clearly $x\in W$ by the assumption of \cref{thm:kappa:case:U2inWandv}
and the fact that $v\notin U_2$.
Also, note that $W^+\cup\set{v}\subset V_0\cup V_1$.
Now, since $u_2\in N(x)\subset M_2(x^+)$,
$u_1\in Y_1\subseteq W^+$, and,
by~\cref{thm:localkappa:eq:d(w1+-w2+)=2}, $W^+\subset M_2(x^+)$,
\begin{equation}
\begin{aligned}
d(u_1)+d(u_2)
&\le\card[\big]{\bigl((V_0\cup V_1)\setminus(W^+\cup\set{v})\bigr)\cap M_3(x^+)}
	+\card{M_3(x^+)\setminus V_1}\\
&=\card{M_3(x^+)\setminus(W^+\cup\set{v})}+\card{V_0\cap M_3(x^+)}\\
&\le\card{M_3(x^+)}-\card{W^+}-1+\kappa_3.
\end{aligned}
\end{equation}
Thus $d(u_1)+d(u_2)+d(v)\le\card{M_3(x^+)}+\kappa_3-1\le\card{M_3(x^+)}+\kappa_3(x^+)-1$, a contradiction.
%
We can conclude that $U_2=\emptyset$.

\medskip
Thus there are only two nonempty sets $U_i$, that is $M_3\open=U_0\cup U_1$
(recall that $U_0$ and $U_1$ must be nonempty).
Since $V_0$ is a minimal cut set,
every vertex of $V_0$ has a neighbor in each of the sets $V_1,\dotsc,V_p$.
Also, clearly $p\ge2$.
Now $M_2(y_0^+)\subseteq M_3\open\subseteq V_0\cup V_1$,
so $y_0^+\in V_1$ and $N(y_0^+)\subseteq V_1$.
Also $v\in V_1$, since $V_0$ is minimal and
$N(v)=W\subseteq{M_3\open}\subseteq V_0\cup V_1$ by \cref{thm:localkappa:eq:Winternal}.
If $z^+\notin V_1$ for some $z\in W$, then $z^+\in V_0$
because, by~\cref{thm:localkappa:eq:d(w1+-w2+)=2}, $z^+\in M_2(y_0^+)$.
In that case, pick $u\in N(z^+)\cap V_2$.
Then
\begin{equation}
\label{thm:localkappa:eq:W+inV1}
\begin{aligned}
d(y_0^+)+d(u)
&\le\card[\big]{\bigl(V_1\setminus(W^+\cup\set{v})\bigr)\cap M_3(z^+)}
	+\card{M_3(z^+)\setminus V_1}\\
&=\card{M_3(z^+)\setminus(W^+\cup\set{v})}+\card{V_0\cap W^+}\\
&\le\card{M_3(z^+)}-\card{W^+}-1+\kappa_3.
\end{aligned}
\end{equation}
so $d(y_0^+)+d(u)+d(v)\le\card{M_3(z^+)}+\kappa_3-1\le\card{M_3(z^+)}+\kappa_3(z^+)-1$, a contradiction.
Thus $W^+\subseteq V_1$.
In the same way we can show that $N(z^+)\subset V_1$ for all $z\in W$:
Assume there is an $a\in N(z^+)\setminus V_1$ for some $z\in W$.
Then $a\in V_0$ since $N(z^+)\subset M_3$.
Pick $u\in N(a)\cap V_2$.
The same calculations as in \cref{thm:localkappa:eq:W+inV1} show again that
$d(y_0^+)+d(u)+d(v)\le\card{M_3(z^+)}+\kappa_3(z^+)-1$, a contradiction.
Thus $N(z^+)\subset V_1$ for all $z\in W$.
In particular, $W\subset V_1$.

Pick $u\in V_2$.
Then $d(y_0^+,u)=3$ since $u\in M_3\setminus M_3\open$.
Pick $w\in N(y_0^+)$ and $a\in N(u)\cap V_0$ so that $wa\in E(G)$,
that is, so that $y_0^+\mskip-2mu wau$ is a path in~$G$.
Clearly $d(a,z^+)\ge2$ for all $z\in W$, since $a\notin V_1$.
In fact, $d(a,z^+)=2$ for all $z\in W$,
since if $d(a,z^+)\ge3$ for some $z\in W$ then
$d(v)+d(z^+)+d(a)\le\card{M_3(y_0^+)}+\card{W}-\card{W^+\cup\set{v,a}}=\card{M_3}-2$.

We shall prove that $d(w,x^+)\le2$ for all $x\in W$.
Assume that $d(w,x^+)\ge3$ for some $x\in W$.
Since $u\notin M_3\open$,
there is an $s\in N(u)\setminus M_3$.
Clearly $d(s,w)\ge3$ since $d(s,y_0^+)\ge4$.
Also $d(s,x^+)\ge3$, since if $d(s,x^+)=2$ then
$d(y_0^+)+d(v)+d(s)\le\card{M_3(x^+)}+\card{W}-\card{W^+\cup\set{v,s}}=\card{M_3(x^+)}-2$,
a contradiction.
But this means that
$d(x^+)+d(w)+d(s)\le\card{M_3(a)}-\card{\set{x^+,w,s}}=\card{M_3(a)}-3$,
a contradiction.
Thus $d(w,x^+)\le2$ for all $x\in W$,
which means that $v\in M_3\open(w)$.
Thus
$d(y_0^+)+d(u)+d(v)\le\card{M_3(w)}+\card{W}-\card{W^+\cup\set{u,v}}=\card{M_3(w)}-2$,
our final contradiction.
We can conclude that $C$ is a Hamilton cycle.
The proof of \cref{thm:localkappa} is complete.
\end{proof}

\begin{remark}
\Cref{thm:localkappa} is not a generalization of \cref{oldthm:kappa}
because there are 2-connected graphs satisfying the conditions of \cref{oldthm:kappa}
where some balls of radius~3 are not 2-connected.
An example of such a graph can be seen in
\cref{fig:threeballnottwoconnected} on page~\pageref{fig:threeballnottwoconnected}.
\end{remark}

\Cref{thm:localkappa} implies the following local analogue of \cref{oldthm:haggkvist}:
\begin{theorem}
\label{thm:localkappaonevertex}
Let $G$ be a connected finite graph on at least three vertices
such that
every ball of radius~3 in $G$ is 2-connected and
$d(u)\ge {1\over 3}\bigl(\card{M_3(v)}+\kappa(G_3(v))\bigr)$
for every vertex $v\in V(G)$
and for each interior vertex $u$ of $G_3(v)$.
Then $G$ is Hamiltonian.
\end{theorem}

Note that the diameter of a graph satisfying the conditions of
\cref{oldthm:haggkvist} or \cref{oldthm:kappa} does not exceed~5 (see \cite{asratian07}).
In contrast with this,
for any integers $n\geq 4$ and $p\geq 3$ there is a $(3p-1)$-regular graph
of diameter~$n$ satisfying the conditions of \cref{thm:localkappa,thm:localkappaonevertex}
(for example, the graph $G(p,2n)$ in \cref{remark:local-but-not-global-ore-graphs}).


\section{A local analogue of Moon--Moser's theorem}
\label{sec:MM}

In this section we will use our localization method to find a local analogue of the following result:
\begin{theorem}[Moon and Moser~\cite{moon63}]
Let $G$ be a balanced bipartite graph with $2n$ vertices, $n\geq 2$,
such that $d(u)+d(v)>n$ for every pair of non-adjacent vertices $u$ and $v$ at odd distance in~$G$.
Then $G$ is Hamiltonian.
\end{theorem}

We will first prove that the result of Moon and Moser is equivalent to the following proposition:
\begin{proposition}
\label{prop:moon-moser-local-equivalent}
Let $G$ be a connected finite balanced bipartite graph on at least four vertices
such that
\[d(u)+d(v)> 1+|N_2(u)|+|N_4(u)|+|N_6(u)|\]
for every vertex $u\in V(G)$
and for each vertex~$v$ of the ball $G_6(u)$ at odd distance $d(u,v)>1$ from~$u$.
Then $G$ is Hamiltonian.
\end{proposition}

\begin{proof}
If $G$ satisfies the conditions of Moon--Moser's theorem then $G$ is connected and $n\geq 1+|N_2(u)|+|N_4(u)|+|N_6(u)|$ for each vertex $u\in V(G)$.
Therefore the conditions of \cref{prop:moon-moser-local-equivalent} hold.

Conversely, suppose that the conditions of \cref{prop:moon-moser-local-equivalent} hold in a balanced bipartite graph~$G$ with $2n$ vertices.
We will show that the diameter of~$G$ does not exceed~$4$.
Suppose to the contrary that the diameter of $G$ is bigger than~$4$. Then there exists a pair of vertices, $u_0$ and $v_0$, with $d(u_0,v_0)=5$.
By the conditions of the proposition,
\begin{equation}
\label{prop:moon-moser-local-equivalent:eq:1}
d(u_0)+d(v_0)> 1+|N_2(u_0)|+|N_4(u_0)|+|N_6(u_0)|.
\end{equation}
Consider a $(u_0,v_0)$-path $u_0u_1u_2v_2v_1v_0$. Since $d(u_0,v_0)=5$, the distance between $u_1$ and $v_1$ is three.
Therefore, by the condition of the proposition,
\begin{equation}
\label{prop:moon-moser-local-equivalent:eq:2}
d(u_1)+d(v_1)> 1+|N_2(u_1)|+|N_4(u_1)|+|N_6(u_1)|.
\end{equation}
Furthermore, $d(u_0,v_0)=5$ implies that
the vertices $u_1$ and $v_0$ have no common neighbors,
and similarly for $v_1$ and $u_0$.
This implies
\begin{equation}
\label{prop:moon-moser-local-equivalent:eq:3}
d(u_1)\leq 1+|N_2(u_0)|+|N_4(u_0)|+|N_6(u_0)|-d(v_0)
\end{equation}
and
\begin{equation}
\label{prop:moon-moser-local-equivalent:eq:4}
d(v_1)\leq 1+|N_2(u_1)|+|N_4(u_1)|+|N_6(u_1)|-d(u_0).
\end{equation}
Therefore
\begin{equation}
\label{prop:moon-moser-local-equivalent:eq:5}
d(u_1)+d(v_0)\leq 1+|N_2(u_0)|+|N_4(u_0)|+|N_6(u_0)|
\end{equation}
and
\begin{equation}
\label{prop:moon-moser-local-equivalent:eq:6}
d(v_1)+d(u_0)\leq 1+|N_2(u_1)|+|N_4(u_1)|+|N_6(u_1)|.
\end{equation}
Then \cref{prop:moon-moser-local-equivalent:eq:5} and \cref{prop:moon-moser-local-equivalent:eq:6} imply
\begin{equation}
\label{prop:moon-moser-local-equivalent:eq:7}
\begin{aligned}
d(u_0)+d(v_0)+d(u_1)+d(v_1)
&\leq 1+|N_2(u_0)|+|N_4(u_0)|+|N_6(u_0)|\\
	&\hphantom{{}={}}+1+|N_2(u_1)|+|N_4(u_1)|+|N_6(u_1)|.
\end{aligned}
\end{equation}
This and \cref{prop:moon-moser-local-equivalent:eq:1} imply
\begin{equation}
d(u_1)+d(v_1)< 1+|N_2(u_1)|+|N_4(u_1)|+|N_6(u_1)|
\end{equation}
which contradicts \cref{prop:moon-moser-local-equivalent:eq:2}.
Thus the diameter of $G$ does not exceed~$4$.
This means that $G_4(u)=G$ and $N_6(u)=\emptyset$ for each vertex $u\in V(G)$.
Therefore the conditions of \cref{prop:moon-moser-local-equivalent} are equivalent
to the conditions of Moon--Moser's theorem.
\end{proof}

One can show that if,
in accordance with step~2 of our localization method,
we replace the condition on the ball $G_6(u)$
in \cref{prop:moon-moser-local-equivalent} with a similar condition on the ball $G_4(u)$,
that is, with the condition
\[d(u)+d(v)>1+|N_2(u)|+|N_4(u)|
\text{ for each vertex~$v$ with $d(u,v)=3$,}\]
we obtain a result that
generalizes Moon--Moser's theorem.
Moving on to step~3, we will prove the following stronger result using the structure of balls:

\begin{theorem}
\label{thm:localMM}
A connected finite balanced bipartite graph $G$ on at least four vertices is Hamiltonian if
\[d(u)+d(v)> 1+|N_2(u)\cup N(v)|\]
for every pair of vertices $u,v\in V(G)$ with $d(u,v)=3$.
\end{theorem}

We begin with a simple lemma.
\begin{lemma}
\label{lem:localMMlemma}
The condition $d(u)+d(v)> 1+|N_2(u)\cup N(v)|$
is equivalent to
\[d(u)\geq 2+|N_2(u)\setminus N(v)|.\]
\end{lemma}

\begin{definition}
Let $ P=v_1v_2\dotsp v_n$ be a maximal path with a given direction (from $v_1$ to $v_n$) in a graph $G$ where $d(v_1)\geq 2$. We define a parameter $f(P)$ as follows:
If $N(v_1)=\{v_{i_1},v_{i_2},\dotsc,v_{i_p}\}$ where $i_1<i_2<\dotsb <i_p$, then $f(P)=i_p$. In particular, $f(P)=n$ if $v_1v_n\in E(G)$.
\end{definition}

\begin{proposition}
\label{prop:localMM-path-not-longer-than-cycle}
Let $G$ satisfy the conditions of \cref{thm:localMM}.
Then the length of any path of $G$ does not exceed the length of a longest cycle of~$G$.
\end{proposition}

\begin{proof}
Let $P=v_1v_2v_3\dotsp v_n$ be a longest path of $G$ such that $f(P)\geq f(P')$ for any other longest path $P'$ of~$G$.
We will show that $f(P)=2t$ where $t=\lfloor\frac{n}{2}\rfloor$.

Let $N(v_1)=\{v_{i_1},v_{i_2},\dotsc,v_{i_p}\}$ where $i_1<i_2<\dotsb <i_p$. Thus $f(P)=i_p$.
Since $P$ is a longest path and $G$ is bipartite, $N(v_1)\subseteq \{v_2,v_4,\dotsc,v_{2t}\}$.
Suppose that $i_p<2t$.
Since $i_p$ is an even integer, $i_p+2\leq 2t$ and $d(v_1,v_{i_p+2})=3$.
Consider the set
\[A=N(v_{i_p+2})\cap \{v_{i_2-1},v_{i_3-1},\dotsc,v_{i_p-1}\}.\]
Clearly, $A\not=\emptyset$ because otherwise $|N_2(v_1)\setminus N(v_{i_p+2})|\geq d(v_1)-1$ and we obtain a contradiction:
\[d(v_1)\geq 2+| N_2(v_1)\setminus N(v_{i_p+2})|\geq 2+d(v_1)-1=d(v_1)+1.\]
Since $A\not=\emptyset$, there is an integer $r, 2\leq r\leq p$, such that $v_{i_r-1}v_{i_p+2}\in E(G)$.
Using the edge $v_{i_r-1}v_{i_p+2}$ we can construct a new longest path~$Q$ with origin $v_{i_p+1}$ and terminus $v_n$ where
$Q=v_{i_p+1}\overleftarrow Pv_{i_r}v_1\overrightarrow Pv_{i_{r-1}}v_{i_p+2}\overrightarrow Pv_n$.
Clearly, $N(v_{i_p+1})\subset V(Q)$, because otherwise one can extend $Q$ (from $v_{i_p+1}$) and obtain a longer path in~$G$.
Therefore, $Q$ is a longest path with $V(Q)=V(P)$ where $f(Q)\geq i_p+2>f(P)$, which contradicts the choice of~$P$.

Thus $f(P)=2t=2\lfloor\frac{n}{2}\rfloor$. This means that $|V(P)|-|V(C)|\leq 1$ for the cycle $C=v_1\overrightarrow Pv_{2t}v_1$,
that is, the length of $P$ does not exceed the length of~$C$.
Therefore the length of $P$ does not exceed the length of a longest cycle of~$G$. The proof of \cref{prop:localMM-path-not-longer-than-cycle} is complete.
\end{proof}

\begin{boldproof}[Proof of \cref{thm:localMM}]
If $d(u,v)\leq 2$ for each pair $u,v\in V(G)$, then $G$ is a complete balanced bipartite graph
and therefore $G$ has a Hamilton cycle.
Suppose that $G$ is a balanced bipartite graph which is not complete.
If $d(u)=1$ for some vertex~$u$,
then since $G$ is connected and balanced
and has at least four vertices
there is a vertex~$v$ such that $d(u,v)=3$.
But then $d(u)+d(v)=1+\card{N(v)}\ngtr1+\card{N_2(u)\cup N(v)}$,
a contradiction.
Thus $d(u)\geq 2$ for each $u\in V(G)$.

Let $C$ be a longest cycle of~$G$. Choose a direction $\overrightarrow C$ of $C$.
We will show that $V(C)=V(G)$,
using the fact that \cref{prop:localMM-path-not-longer-than-cycle}
implies that $G$ cannot contain a path with more than $\card{V(C)}+1$ vertices.
Suppose that $V(C)\not=V(G)$. Then every component of the graph $G-V(C)$ consists of an isolated vertex,
because otherwise there would be a vertex~$x$ on $C$ and two vertices, $y$ and $z$, outside of $C$ such that $xy, yz\in E(G)$ and the path $zyx\overrightarrow Cx^-$ has
$|V(C)|+2$ vertices.

Let $(V_1,V_2)$ be the bipartition of~$G$. Choose two vertices in $G-C$, $x_1\in V_1$ and $u_1\in V_2$, and two of their neighbors, $x_2$ and $u_2$, on $C$ such that $x_1x_2\in E(G)$,
$u_1u_2\in E(G)$, and no vertex in $u_2^+\overrightarrow Cx_2^-$ is adjacent to $x_1$ or $u_1$.
Clearly, $u_2^+\not=x_2$ because otherwise the
path $x_1x_2\overrightarrow Cu_2u_1$ has $|V(C)|+2$ vertices. Furthermore, $u_1x_1\notin E(G)$ because every component in $G-C$ consists of one vertex.
Denote the path $x_1x_2\overrightarrow Cu_2u_1$ by~$P$ and the path $u_2^+\overrightarrow Cx_2^-$ by~$B$. The origin $u_2^+$ of~$B$ we denote by $b_1$.
Since $x_1,u_2\in V_1$ and $u_1,x_2,b_1\in V_2$, the paths $P$ and $B$ has odd lengths.

We will construct a sequence $(P_1,B_1), (P_2,B_2),\dotsc,$ where $P_i$ and $B_i$ are two disjoint paths of odd length with
$V(P_i)\cup V(B_i)=V(C)\cup \{u_1,x_1\}$,
for every $i\geq 1$, which satisfy the following conditions:
\begin{enumerate}
\item $V(P_1)\subseteq V(P_2)\subseteq V(P_3)\subseteq \dotsb$
and $V(B_1)\supseteq V(B_2)\supseteq V(B_3)\supseteq\dotsb$,
\item the origin of $P_k$ belongs to $V_1$, the terminus of $P_k$ is $u_1$ and the second to last vertex of $P_k$ is $u_2$, for every $k\geq 1$,
\item the origin of $B_k$ is $u_2^+$ and the terminus belongs to $V_1$ for every $k\geq 1$,
\item the origin of $P_k$ is not adjacent to the origin of $B_k$ for any $k\geq 1$,
\item if $V(P_{k+1})=V(P_k)$ then $f(P_{k+1})>f(P_k)$, $k\geq 1$.
\end{enumerate}

Put $P_1=P$ and $B_1=B$. Suppose that we have constructed the pairs
$(P_1,B_1),\dotsc,(P_k,B_k)$, $k\geq 1$, which satisfy the conditions (i)--(v).
Let $P_k=v_1v_2\dotsp v_{2m}$
where $v_1\in V_1$, $v_{2m}=u_1$, $v_{2m-1}=u_2$,
and let
$B_k=b_1b_2\dotsp b_{2s}$, where
$s\geq 1$ and $b_1\in V_2$.
(Clearly, $m\geq 3$ because if $m=2$ then
$V(P_k)=\set{x_1,x_2,u_1,u_2}$,
so the path $x_1x_2x_2^-\overleftarrow Bu_2^+u_2u_1$ in~$G$
has $|V(C)|+2$ vertices.)

Let $N(v_1)=\{v_{i_1},v_{i_2},\dotsc,v_{i_p}\}$ where $i_1<i_2<\dotsb <i_p\leq 2m$.
Clearly, $i_p\not=2m$, that is, the origin $v_1$ of $P_k$ is not adjacent to the terminus $v_{2m}$ because otherwise, by (ii), $v_{2m-1}=u_2$ and $G$ contains the path
$v_{2m}v_1\overrightarrow P_kv_{2m-1}b_1\overrightarrow B_kb_{2s}$ with $|V(C)|+2$ vertices.
Thus $i_p<2m$. Then $d(v_1,v_{i_p+2})=3$ and the same arguments as in the proof of \cref{prop:localMM-path-not-longer-than-cycle} imply that there is
an integer $r, 2\leq r\leq p,$ such that $v_{i_r-1}v_{i_p+2}\in E(G)$.

Consider a path~$Q$ with origin $v_{i_p+1}$ and terminus $v_{2m}=u_1$ where
\[Q=v_{i_p+1}\overleftarrow P_kv_{i_r}v_1\overrightarrow P_kv_{i_{r-1}}v_{i_p+2}\overrightarrow P_kv_{2m}.\]
Clearly, $v_{i_p+1}\in V_1$ and $v_{i_p+1}$ has no neighbor in $V(G)\setminus V(C)$ because otherwise
if $zv_{i_p+1}\in E(G)$ and $z\in V(G)\setminus V(C)$ then, by property (ii),
$v_{2m-1}=u_2$ and $G$ contains the path
$zv_{i_p+1}\overrightarrow Qv_{2m-1}b_1\overrightarrow B_kb_{2s}$
with $|V(C)|+2$ vertices.
Therefore only the following two cases are possible:
\begin{case}
$v_{i_p+1}$ has a neighbor in~$B_k$.
\end{case}
Since $b_1\in V_2$ and $b_{2s}, v_{i_p+1}\in V_1$, the vertex $v_{i_p+1}$ is adjacent to a vertex $b_{2\ell+1}$ for some $\ell\geq 1$.
(Note that $\ell\not=0$ because otherwise $G$ contains the path $b_{2s}\overleftarrow B_kb_1v_{i_p+1}\overrightarrow Qv_{2m}$
with $|V(C)|+2$ vertices.)
Then we extend $P_k$ to $P_{k+1}$ and shorten $B_{k}$ to $B_{k+1}$ as follows:
\[P_{k+1}=b_{2s}\overleftarrow B_kb_{2\ell+1}v_{i_p+1}\overrightarrow Qv_{2m}
\text{ and }
B_{k+1}=b_1b_2\dotsp b_{2\ell}.\]
\begin{case}
$v_{i_p+1}$ has no neighbors in~$B_k$.
\end{case}
Then put $P_{k+1}=Q$ and $B_{k+1}=B_k$. We have $V(P_k)=V(P_{k+1})$ and $f(P_{k+1})>f(P_k)$, since $f(P_k)=i_p$ and $f(P_{k+1})\geq i_p+2$.

It follows from properties (i)--(v) that after a finite number of steps we will obtain a pair $(P_q,B_q)$ of disjoint paths of odd length with $|V(P_q)|+|V(B_q)|=|V(C)|+2$,
where the origin of $B_q$ is $b_1$, $P_q$ is
a maximal path of odd length with terminus $u_1$ and the second to last vertex $u_2$, and, moreover, $f(P_q)=|V(P_q)|$. This means that the origin of $P_q$ (we denote it by $a_1$) is adjacent to the terminus $u_1$.
Let $b_{2r}$ denote the terminus of~$B_q$. By property (ii), $u_2$ is the second to last vertex of $P_q$. Moreover, $u_2b_1\in E(G)$. But then $G$ contains the path
$u_1a_1\overrightarrow P_qu_2b_1\overrightarrow B_qb_{2r}$
with $\card{V(C)}+2$ vertices, a contradiction. The proof of \cref{thm:localMM} is complete.
\end{boldproof}

\begin{remark}
\label{rem:localMMdiam6}
The diameters of graphs satisfying the conditions of Moon--Moser's theorem do
not exceed~4. In contrast with this there is an infinite class of
bipartite graphs of diameter~$6$ that satisfy the condition
of \cref{thm:localMM}. Consider, for example, the graph $G(n)$ on $5n$
vertices, $n\geq 5$, which is defined as follows: its vertex set is
$\bigcup_{i=0}^6 V_i$, where $V_0,V_1,\dotsc,V_{6}$ are pairwise disjoint
sets of cardinality $|V_0|=|V_6|=1$, $|V_1|=|V_{5}|=n-1$, $|V_{i}|=n$ for $i=2,3,4$ and two vertices $x,y\in V(G(n))$ are
adjacent if and only if $x\in V_i$ and $y\in V_{i+1}$, for $i=0,1,2,3,4,5.$
\end{remark}

\begin{proposition}
The diameter of a graph satisfying the conditions of \cref{thm:localMM} can not exceed~6.
\end{proposition}
\begin{proof}
Let $G$ be a graph satisfying the conditions of \cref{thm:localMM},
and assume that $v$ and $v'$ are two vertices in~$G$ with $d(v,v')=7$.
Pick $u\in N_3(v)$ and $u'\in N_3(v')$ such that $uu'\in E(G)$.
Now, by using \cref{lem:localMMlemma} and the fact that $G$ is bipartite we get
\newcommand{\myeqdiff}{\hphantom{{}=1+d(u')}}
\begin{align*}
d(u)
&\ge2+\card{N_2(u)\setminus N(v)}\\
&=2+\card[\big]{N_2(u)\cap\bigl(N_3(v)\cup N_5(v)\bigr)}\\
&\ge1+\card[\big]{N(u')\cap\bigl(N_3(v)\cup N_5(v)\bigr)}\\
&=1+d(u')
\ge3+\card{N_2(u')\setminus N(v')}\\
&\myeqdiff=3+\card[\big]{N_2(u')\cap\bigl(N_3(v')\cup N_5(v')\bigr)}\\
&\myeqdiff\ge2+\card[\big]{N(u)\cap\bigl(N_3(v')\cup N_5(v')\bigr)}\\
&\myeqdiff=2+d(u).
\end{align*}
This is clearly a contradiction.
Thus the diameter of $G$ does not exceed~6.
\end{proof}


\section{Extensions to infinite locally finite graphs}
\label{sec:infinite}

Extending the notion of cycles, Diestel and K\"uhn~\cite{diestel04a}
defined \emph{circles} in~$|G|$ as the image of a homeomorphism which maps
the unit circle~$S^1\subset\mathbb{R}^2$ to~$|G|$.
The graph~$G$ is called Hamiltonian if there exists a circle in~$|G|$
that contains all vertices and ends of~$G$.
Such a circle is called a \emph{Hamilton circle} of~$G$.
For finite graphs~$G$, this coincides with the usual meaning,
namely that there is a Hamilton cycle of~$G$.

Kündgen, Li, and Thomassen~\cite{kundgen17} introduced
another concept for infinite locally finite graphs:
A closed curve
in the Freudenthal compactification~$|G|$ is called a \emph{Hamiltonian curve} of~$G$
if it meets every vertex of~$G$ exactly once (and hence it meets every end at least once).
They proved the following theorem:
\begin{theorem}[Kündgen, Li, and Thomassen~\cite{kundgen17}] 
\label{thm:kundgen17}
The following are equivalent for any locally finite graph~$G$.
\begin{enumerate}
\item For every finite vertex set~$S$, $G$~has a cycle containing~$S$.
\item $G$ has a Hamiltonian curve.
\end{enumerate}
\end{theorem}

\Cref{thm:kundgen17} gives a possibility to extend different results for finite
graphs to infinite locally finite graphs.
In \cite{asratian18b}, the authors used it to prove the following theorems:

\begin{theorem}[\cite{asratian18b}]
\label{thm:infinitechvatal}
Let $r$ be a positive integer and $G$ a connected, infinite, locally finite graph
such that $\kappa(G_r(u))\geq \alpha(G_{r+1}(u))$ for every vertex $u\in V(G)$.
Then $G$ has a Hamiltonian curve.
\end{theorem}

\begin{theorem}[\cite{asratian18b}]
\label{localjung1}
Let $G$ be a connected, infinite, locally finite graph where
every ball of radius~2 is 2-connected and
$d(u)+d(v)\geq |M_2(x)|-1$ for every path $uxv$ with $uv\notin E(G)$.
Then $G$ has a Hamiltonian curve.
\end{theorem}

We will use \cref{thm:kundgen17} to extend \cref{thm:localkappa,thm:localkappaonevertex}
to infinite locally finite graphs.
Although \cref{thm:localkappa} was proved by contradiction,
the proof can alternatively be outlined as follows:
We start with a cycle~$C$.
If it is not Hamiltonian,
then we extend~$C$ to a longer cycle~$C'$.
We repeat this extension procedure until the cycle becomes Hamiltonian.
The same idea can be used to construct a cycle that contains a given finite vertex set $S$
in an
infinite locally finite graph $G$,
although we require a bit more control over the extensions.

\begin{theorem}
\label{thm:infinitekappa}
Let $G$ be a connected, infinite, locally finite graph
such that
every ball of radius~3 in $G$ is 2-connected and
$d(x)+d(y)+d(z)\ge\card{M_3(v)}+\kappa\bigl(G_3(v)\bigr)$
for every vertex $v\in V(G)$ and
for all triples of independent interior vertices $x,y,z$ of $G_3(v)$.
Then $G$ has a Hamiltonian curve.
\end{theorem}

\begin{proof}
Let $S$ be a finite vertex set in~$G$.
We will show that there is a cycle of~$G$ containing all vertices of~$S$.
Pick a vertex $a\in S$ and let $r=\max_{x\in S}d(a,x)$.
Among all cycles in $G_{r+12}(a)$ that contain a vertex of $G_{r+1}(a)$,
let $C$ be a longest one.
Assume that $C$ does not contain all vertices of~$S$.
Then there is at least one vertex in $M_r(a)\setminus V(C)$ with a neighbor on $C$.
We will show that this leads to a contradiction.

\begin{case}
There is a vertex $v\in M_r(a)\setminus V(C)$ with $N(v)\subseteq V(C)$.
\end{case}
By repeating the arguments in
the proof of \cref{thm:localkappa},
we can extend~$C$ to a longer cycle $C'$ such that $V(C)\cup \{v\}\subseteq V(C')$.
Then $C'$ contains a vertex from $G_{r+1}(a)$ since $C$ contains~$v$.
An analysis of the proof of \cref{thm:localkappa} shows that
$C'$ only differs from~$C$ inside $G_3(v)$, that is, $C'$ is in $G_{r+12}(a)$.
Thus $C'$ is longer than $C$, is in $G_{r+12}(a)$, and contains a vertex from $G_{r+1}(a)$, a contradiction.

\begin{case}
Each vertex in $M_r(a)\setminus V(C)$ has a neighbor outside of $C$.
\end{case}
Let $u$ be a vertex in $M_r(a)\setminus V(C)$ with a neighbor on $C$ and $N(u)\nsubseteq V(C)$.
Using \cref{lem:localbondythm} we can extend~$C$ to a longer cycle~$C'$
that differs from~$C$ only inside~$G_{12}(u)$
and contains a neighbor of~$u$,
that is, $C'$ is contained in $G_{r+12}(a)$
and contains a vertex of~$G_{r+1}(a)$,
a contradiction.

\medskip
We can conclude that $C$ contains~$S$.
\Cref{thm:kundgen17} now implies that $G$ has a Hamiltonian curve.
\end{proof}

This result implies an extension of \cref{thm:localkappaonevertex}:

\begin{theorem}
\label{thm:infinitelocalkappaonevertex}
Let $G$ be a connected, infinite, locally finite graph
such that
any ball of radius~3 in $G$ is 2-connected and
$d(u)\ge {1\over 3}\bigl(\card{M_3(v)}+\kappa(G_3(v))\bigr)$
for every vertex $v\in V(G)$
and for each interior vertex $u$ of $G_3(v)$.
Then $G$ has a Hamiltonian curve.
\end{theorem}

\pagebreak[2]
Finally, we suggest the following conjecture:
\begin{conjecture}
Let $G$ be a connected, infinite, locally finite graph such that
every ball of radius~3 in $G$ is 2-connected and
\[d(x)+d(y)+d(z)\ge\card{M_3(v)}+\kappa\bigl(G_3(v)\bigr)\]
for every vertex $v\in V(G)$ and
for all triples of independent interior vertices $x,y,z$ of $G_3(v)$.
Then $G$ has a Hamilton circle.
\end{conjecture}

\section*{Acknowledgment}
The authors thank Carl Johan Casselgren for helpful suggestions on this manuscript.


\end{document}